\documentclass[11pt,reqno,twoside]{amsart}
\usepackage{amsfonts,amsmath,amssymb}
\usepackage{mathrsfs,mathtools}
\usepackage{enumerate}
\usepackage{hyperref}
\usepackage{esint}
\usepackage{graphicx}
\usepackage{bm}
\usepackage{commath}
\usepackage{esint}
\DeclareMathAlphabet{\mathpzc}{OT1}{pzc}{m}{it}

\usepackage{placeins} 
\usepackage{flafter} 

\usepackage{tikz}

\usepackage[textsize=small]{todonotes}
\definecolor{darkgreen}{RGB}{0,95,0}

\numberwithin{equation}{section}

\usepackage[dvips,bottom=1.4in,right=1in,top=1in, left=1in]{geometry}

\makeatletter
\def\eqnarray{\stepcounter{equation}\let\@currentlabel=\theequation
\global\@eqnswtrue
\tabskip\@centering\let\\=\@eqncr
$$\halign to \displaywidth\bgroup\hfil\global\@eqcnt\z@
  $\displaystyle\tabskip\z@{##}$&\global\@eqcnt\@ne
  \hfil$\displaystyle{{}##{}}$\hfil
  &\global\@eqcnt\tw@ $\displaystyle{##}$\hfil
  \tabskip\@centering&\llap{##}\tabskip\z@\cr}

\def\endeqnarray{\@@eqncr\egroup
      \global\advance\c@equation\m@ne$$\global\@ignoretrue}

\newtheorem{theorem}{Theorem}[section]
\newtheorem{corollary}[theorem]{Corollary}
\newtheorem{definition}[theorem]{Definition}
\newtheorem{lemma}[theorem]{Lemma}

\newtheorem{proposition}[theorem]{Proposition}
\newtheorem{assumption}[theorem]{Assumption}
\newtheorem{remark}[theorem]{Remark}
\numberwithin{equation}{section}

\def\RR{{\mathbb{R}}}

\def\Om{\Omega}
\def\om{\omega}

\def\om{\omega}

\def\ga{\gamma}

\def\pOm{\partial\Omega}

\begin{document}

\title[An Algorithm for Optimal Control of Fractional Problems with State Constraints]{Moreau-Yosida Regularization for Optimal Control of Fractional Elliptic Problems with State and Control Constraints}

\author{Harbir Antil}
\address{Department of Mathematical Sciences, George Mason University, Fairfax, VA 22030, USA.}
\email{hantil@gmu.edu}
\author{Thomas S. Brown}
\address{Department of Mathematical Sciences, George Mason University, Fairfax, VA 22030, USA.}
\email{tbrown62@gmu.edu}
\author{Deepanshu Verma}
\address{Department of Mathematical Sciences, George Mason University, Fairfax, VA 22030, USA.}
\email{dverma2@gmu.edu}

\thanks{This work is partially supported by 
the Air Force Office of Scientific Research under Award NO: FA9550-19-1-0036 and 
NSF grants DMS-1818772 and DMS-1913004. }

\date{\today}

\begin{abstract}
Recently in \cite{HAntil_DVerma_MWarma_2019b}, the authors have studied a state and control constrained optimal 
control problem with fractional elliptic PDE as constraints. The goal of this paper is to continue that program 
forward and introduce an algorithm to solve such optimal control problems. We shall employ the well-known Moreau-Yosida regularization to handle the state constraints. Similarly to the classical case, we establish the convergence (with rate) of the regularized control problem to the original one. We discretize the problem using a finite element method and establish convergence of our numerical scheme. We emphasize that due to the non-smooth nature of the fractional Laplacian, the proof for the classical case does not apply to the fractional case. The numerical experiments confirm all our theoretical findings. 
\end{abstract}


\maketitle

\section{Introduction} \label{Sec:1}

Let $\Om \subset \mathbb{R}^N$ ($N \ge 1$) be a bounded Lipschitz domain satisfying the exterior cone condition with boundary 
$\pOm$.  The main goal of this paper is develop a Moreau-Yosida regularization based algorithm to solve a state and control constrained optimal control problem.  For Banach spaces $U$ and $Z$, a desired state $u_d \in L^2(\Omega)$ and penalization parameter $\alpha >0$, the problem can be stated as 
\begin{subequations}\label{eq:dcp}
 \begin{equation}\label{eq:Jd}
    \min_{(u,z)\in (U,Z)} J(u,z):=\frac{1}{2}\|u-u_d\|_{L^2{(\Om)}}^2 + \frac{\alpha}{2} \|z\|^2_{L^2(\Om)} ,
 \end{equation}
 subject to the fractional elliptic PDE with $0<s<1$: Find $u \in U$ solving 
 \begin{equation}\label{eq:Sd}
 \begin{cases}
    (-\Delta)^s u &= z \quad \mbox{in } \Om, \\
                u &= 0 \quad \mbox{in } \RR^N\setminus \Om  ,
 \end{cases}                
 \end{equation} 
with the state constraints 
 \begin{equation}\label{eq:Ud}
    u|_{\Omega} \in \mathcal{K} := \left\{ w \in C_0(\Omega) \ : \ w(x) \le u_b(x) , 
        \quad \forall x \in \overline{\Om}  \right\} .
 \end{equation}
Here $C_0(\Om)$ is the space of continuous functions in $\overline\Om$ that vanish
on $\pOm$ and $u_b \in C(\overline\Om)$. By extending functions by zero outside $\Omega$, we can identify $C_0(\Omega)$ with the space $\{u \in C_c(\RR^N) \ : \ u = 0 \mbox{ in } \RR^N\setminus\Om \}$. 
In addition, we assume the control constraints
 \begin{equation}\label{eq:Zd}
    z \in Z_{ad} \subset 
    L^p(\Om), 
 \end{equation}
where $Z_{ad}$ is a non-empty, closed, and convex set and we require that the real number $p$ satisfies
 \end{subequations}
\begin{equation}\label{cond-p}
\begin{cases}
p>\frac{N}{2s}\;\;&\mbox{ if }\; N>2s,\\
p>1 \;\;&\mbox{ if }\; N=2s,\\
p=1\;\;&\mbox{ if }\; N<2s.
\end{cases}
\end{equation} 
Notice that the condition on $p$ in \eqref{cond-p} implies that the solution to the state
equation \eqref{eq:Sd} is in $L^\infty(\Om)$. In addition, the exterior cone condition on 
the domain $\Omega$ guarantees continuity of such solutions \cite[Theorem~3.4]{HAntil_DVerma_MWarma_2019b}. 
As a result, the Lagrange multiplier corresponding to the inequality constraint in \eqref{eq:Ud} 
is a Radon measure which is much more tractable than the dual space of $L^\infty(\Om)$.

Using first principle arguments and a constitutive relationship, the article \cite{CWeiss_BvBWaanders_HAntil_2018a} 
has recently derived a fractional Helmholtz equation. In addition, it shows a qualitative match 
between the real datum and numerical experiments. See also \cite{HAntil_CNRautenberg_2019b} for 
applications in imaging science. Motivated by such emerging applications, the optimal control of fractional 
PDEs with control constraints has recently received a tremendous amount of attention, we refer to 
\cite{antil2017optimal,MDElia_MGunzburger_2014a} and the references therein.  

However, the only existing work that provides a complete theoretical analysis for optimal control problem with state
constraints is \cite{HAntil_DVerma_MWarma_2019b}. This work has introduced several new theoretical tools
which requires a finer analysis than the classical case of $s=1$, for instance, characterization of the dual 
of the fractional order Sobolev spaces. This characterization has helped establish Sobolev regularity of 
fractional PDEs with measure valued datum, which is the case for the adjoint equation.

However, there are no existing approximation schemes and numerical algorithms to solve \eqref{eq:dcp}. 
As we have established in our previous work \cite{HAntil_DVerma_MWarma_2019b}, caution must be observed 
in claiming that one can apply existing techniques to solve the above problem. Indeed, 
in this paper we observe that the existing proofs, due to the nonlocal and non-smooth 
nature of fractional Laplacian, cannot be applied to show the convergence of numerical scheme.

The current 
paper aims to introduce a Moreau-Yosida regularization based solution algorithm for \eqref{eq:dcp} first 
in function spaces and then introduces a finite element discretization and discusses convergence of this
scheme. We emphasize that the use of Moreau-Yosida regularization within optimal control context is not new. There are many existing works on optimal control problems that use such a regularization, see for instance, the monograph \cite{KIto_KKunisch_2008a} and articles \cite{HiHi2009,hintermuller_kunisch}.  Other (incomplete) list of works that have applied such a regularization include shape optimizaiton (\cite{KeUl2015}), nonlinear or semilinear problems (\cite{TrVo2001,UlUlBr2017}), problems involving a nonlocal radiation condition (\cite{MeYo2009}), path following methods (\cite{HiScWo2014,KuWa2012}), see also for augmented Lagrangian techniques (\cite{KaStWa2018}).  

The remainder of the paper is organized as follows: In section~\ref{Sec:2}, we introduce 
some notation and preliminary results. In section~\ref{s:rocp}, we discuss the 
Moreau-Yosida regularized optimal control problem. Section~\ref{s:convergence}
is dedicated to the convergence (with rate) of this regularized problem to the original 
problem. In section~\ref{s:fem} we discuss the finite element discretization of the 
regularized optimal control problem and we establish the convergence of this fully discrete 
problem to the continuous one. At first, we show regularization parameter dependent
convergence and later (for a certain range of $s$), we show convergence of the scheme 
independent of this regularization parameter. Finally, in section~\ref{s:numerics}, 
we conclude with several numerical examples that confirm our theoretical findings.

\section{Notation and Preliminaries}
\label{Sec:2}
 
In this section we introduce some notation and review some preliminary results, see
also \cite{HAntil_DVerma_MWarma_2019b}. In all that follows, unless otherwise stated, we will 
take  $\Om$ as in the previous section,  $0 < s < 1$,  and $1 \le p < \infty$. We dedicate 
subsection~\ref{s:funspaces} to function spaces, this is followed by subsection~\ref{s:nsoln} which 
contains the notion of solution to the state equation \eqref{eq:Sd}. Finally, in Subsection~\ref{s:ocp} 
we discuss some results for the optimal control problem \eqref{eq:dcp}.

\subsection{Function spaces and fractional Laplacian}
\label{s:funspaces}
For a sufficiently regular $u$ defined on $\RR^N$, we let  
    \[
        D_{s,p} u[x,y] :=\frac{u(x)-u(y)}{|x-y|^{{\frac{N}{p}}+s}}
    \]
define a map on $\RR^N \times \RR^N$.  With this function, we define the Sobolev space  
    \[
        W^{s,p}(\Om) := \Big\{ u \in L^p(\Om) \ : \ D_{s,p} u \in 
                L^p(\Om \times \Om)  \Big\}, 
    \]
which we endow with the norm  
 \[
    \|u\|_{W^{s,p}(\Om)} := \left(\|u\|_{L^p(\Omega)}^p
        + \|D_{s,p} u\|_{L^p(\Om\times\Om)}^p \right)^{\frac1p}.      
 \] 
Using $\mathcal{D}(\Om)$ to denote the space of smooth functions with compact support in $\Omega$, we let
\begin{equation*}
 W_{0}^{s,p}(\Omega ):=\overline{\mathcal{D}(\Omega )}^{W^{s,p}(\Omega )}.
\end{equation*}
For a relation between $W^{s,p}_0(\Om)$ and classical $W^{s,p}(\Omega)$ space, we refer to 
\cite[Theorem~2.1]{HAntil_DVerma_MWarma_2019b}.

For the Dirichlet problem \eqref{eq:Sd} we also need to consider the Sobolev space 
    \[
        \widetilde{W}^{s,2}_0(\Om) := \left\{ u \in W^{s,2}(\RR^N) \;:\; u = 0 \mbox{ in } 
                \RR^N\setminus\Om \right\} . 
    \]    
In this case, for $0 < s < 1$, we have that
    \begin{align*}
        \|u\|_{\widetilde{W}_0^{s,2}(\Om)}:=\| D_{s,2}u\|_{L^2(\RR^N\times \RR^N)} 
    \end{align*}
defines a norm on $\widetilde{W}_0^{s,2}(\Om)$.     

    \begin{remark}\label{remark}
    {\rm
         From \cite[Remark~2.2]{HAntil_DVerma_MWarma_2019b} we recall that 
            if $p$ satisfies \eqref{cond-p}, then we have
            $\widetilde{W}_0^{s,2}(\Om)\hookrightarrow L^{p'}(\Omega)$, where $p':=\frac{p}{p-1}$.          
    }    
    \end{remark}
    
The following theorem contains an interesting result for $\widetilde{W}_0^{s,p}(\Om)$, 
see \cite[Theorem~2.3]{antil2018b} for a proof.     
\begin{theorem}
\label{thm:gris2}
Given a bounded open set $\Omega\subset\RR^N$ with a Lipschitz continuous boundary and $1<p<\infty$. 
If $\frac 1p<s<1$, then $\widetilde{W}_0^{s,p}(\Om)=W_0^{s,p}(\Omega)$ with equivalent norms. 
\end{theorem}
In the case that $\frac{1}{p} < s < 1$, the previous result allows us to use the norm  
    \[
        \Vert u\Vert_{\widetilde{W}_{0}^{s,p}(\Om)}=\|D_{s,p} u\|_{L^p(\Om\times\Om)} . 
    \]
For $0<s<1$ and $p\in (1,\infty )$, the space $
\widetilde{W}^{-s,p^{\prime }}(\Om)$ is considered to be the dual of 
$\widetilde{W}_{0}^{s,p}(\Om )$, i.e., 
$ \widetilde{W}^{-s,p^{\prime }}(\Om ):=(\widetilde{W}_{0}^{s,p}(\Om ))^{\star }$, where as before we have $p^{\prime }:=\frac{p}{p-1}$.

In the sense of distributions, the {\bf fractional Laplacian} $(-\Delta )^{s}$  is defined by the formula
\[
(-\Delta )^{s}u(x)=C_{N,s}\mbox{P.V.}\int_{\RR^N}\frac{u(x)-u(y)}{|x-y|^{N+2s}}dy, 
\]
where the normalization constant $C_{N,s}$ is given by 
    \[
        C_{N,s}:=\frac{s2^{2s}\Gamma\left(\frac{2s+N}{2}\right)}{\pi^{\frac
        N2}\Gamma(1-s)},
    \]
with $\Gamma $ denoting the standard Euler Gamma function and P.V. denoting the Cauchy principal value of the integral about $x$ (see, e.g. \cite{Caf3,NPV,War}).

We also define the realization of the fractional Laplace operator $(-\Delta)^s$ in $L^2(\Om)$ 
that incorporates the Dirichlet exterior condition $u=0$ in $\RR^N\setminus\Omega$.  This operator is denoted as $(-\Delta)_D^s$, acts on functions $u$ in 
\[
D((-\Delta)_D^s):=\Big\{u|_{\Om},\; u\in \widetilde{W}_0^{s,2}(\Om):\; (-\Delta)^su\in L^2(\Omega)\Big\},
\]
and is given by 
\[
(-\Delta)_D^s(u|_{\Om})=(-\Delta)^su\;\mbox{ in }\;\Omega.
\]
Thus \eqref{eq:Sd} can now be rewritten as 
\begin{equation}\label{CP}
(-\Delta)_D^s u=z\quad \mbox{ in } \Om.
\end{equation}
Next, we recall the integration-by-parts formula for $(-\Delta)^s$ (see \cite{SDipierro_XRosOton_EValdinoci_2017a} for example). 
\begin{proposition}[The integration by parts formula for $(-\Delta)^s$]
If $u \in D((-\Delta)_D^s)$, then for every $v \in \widetilde{W}_0^{s,2}(\Om)$ we have
\begin{equation}\label{intbyParts}
\mathcal{E}(u,v) :=\frac{C_{N,s}}{2} 
        \hspace{-3pt} \int_{\RR^N}\int_{\RR^{N}} 
         \frac{(u(x)-u(y))(v(x)-v(y))}{|x-y|^{N+2s}} \;dxdy = \int_\Om v(-\Delta)^s u \; dx. 
\end{equation} 
\end{proposition}

\subsection{Notions of solution to the state equation}
\label{s:nsoln}

Next we state the notion of weak solution to \eqref{CP}.  In what follows, as well as subsequently throughout, we will use $\langle \cdot, \cdot \rangle_{X^\star, X}$ to represent the duality pairing of elements of a Banach space $X$ and its dual $X^\star$.  
 \begin{definition}[\bf Weak solution] 
 \label{def:weak_d}
    Let $z \in\widetilde{W}^{-s,2}(\Om)$. 
    A $u \in \widetilde{W}^{s,2}_0(\Om)$ 
    is said to be a weak solution to \eqref{eq:Sd} if 
    \[
     \mathcal{E}(u,v)
         = \langle z , v \rangle_{\widetilde{W}^{-s,2}(\Om),\widetilde{W}^{s,2}_0(\Om)}  \qquad \forall v\in \widetilde{W}^{s,2}_0(\Om). 
    \]
 \end{definition}

We finish this section by defining a very-weak solution to \eqref{CP}.  To this end, we introduce the notation for the dual space $C_0({\Om})^\star = \mathcal{M}({\Om})$ 
	where $\mathcal{M}({\Om})$ denotes the space of all Radon measures on ${\Om}$ such that
    \[
        \langle \mu , v \rangle_{\mathcal{M}({\Om}),C_0({\Om})} 
        = \int_{{\Om}} v\;d\mu , 
        \quad \mu \in \mathcal{M}({\Om}), \quad v \in C_0({\Om}),
    \]
and we have the norm
       $ \|\mu\|_{\mathcal{M}({\Om})} = \sup_{v \in C_0({\Om}), \; |v|\le 1} \int_{{\Om}} v\;d\mu $. 

%
\begin{definition}[very-weak solutions]
    \label{def:vwsoln}
        Let $p$ be as in \eqref{cond-p}, $p' = \frac{p}{p-1}$, and $\mu \in \mathcal{M}({\Om})$. A function $u \in L^{p'}(\Om)$ is said to be a very-weak solution
        to \eqref{eq:Sd} if the identity 
        \[
            \int_\Om u (-\Delta)^s v\;dx 
             = \int_\Om v \; d\mu,
        \]
        holds for every $v \in \Big\{ C_0({\Om}) \cap \widetilde{W}^{s,2}_0(\Om) \; : \;
         (-\Delta)^s v \in L^p(\Om) \Big\}$. 
\end{definition}

Existence of a unique weak solution according to the Definition~\ref{def:weak_d} is due to classical 
Lax-Milgram Theorem. For the existence and uniqueness of a very-weak solution, according to the 
Definition~\ref{def:vwsoln}, we refer to \cite[Theorem~3.6]{HAntil_DVerma_MWarma_2019b}. In addition,
\cite{HAntil_DVerma_MWarma_2019b} establishes continuity of solution to the state equation \eqref{eq:Sd}. 
Moreover, it provides well-posedness of fractional PDEs with measure valued datum. The latter is 
essential to establish a Sobolev regularity of the adjoint equation.

\subsection{Optimal Control Problem}
\label{s:ocp}

As seen in section~\ref{s:funspaces}, $(-\Delta)^s_D$ is the realization in $L^2(\Om)$
of the fractional Laplacian $(-\Delta)^s$ that incorporates zero exterior Dirichlet
condition.  From the integration-by-parts formula \eqref{intbyParts}, we can see that it is a self-adjoint operator on $L^2(\Om)$.  We introduce the relevant function spaces
    \begin{equation}\label{eq:funcZU}
    \begin{aligned}
        Z &:= L^p(\Om),  \quad \mbox{with $p$ as in \eqref{cond-p} but } 2 \le p < \infty , \\
        U &:= 
            \{u\in \widetilde{W}_0^{s,2}(\Om)\cap C_0(\Om): (-\Delta)^s_D (u|_\Om) \in L^{p}(\Om)\}.
    \end{aligned}    
    \end{equation}
Notice that previously in the paper we considered $1 \le p$. However, in the definition of $Z$ in \eqref{eq:funcZU} we have taken $2 \le p$. The reason for this  change is that our objective function in \eqref{eq:Jd} contains the $L^2$-regularization on the control $z$ with $\alpha > 0$. Therefore, our controls are at least $L^2$-regular. 
    
The problem \eqref{eq:dcp} can be rewritten as
	\begin{equation}\label{eq:Sdd} 
	\begin{aligned}
		&\min_{(u,z)\in (U,Z)} J(u,z) \\ 
		\text{subject to}& \\
		& (-\Delta)_D^s u = z , \quad \mbox{in } \Om \\
		&u|_{\Om} \in \mathcal{K} \quad \mbox{and} \quad 
		z \in Z_{ad}  .
	\end{aligned}	
	\end{equation}

Notice that for every $z \in Z$, due to \cite[Theorem~3.4]{HAntil_DVerma_MWarma_2019b}, there is a 
unique $u \in U$ that solves the state equation \eqref{eq:Sd}. Using this fact, 
the control-to-state (solution) map
    \[
        S : Z \rightarrow U,\;\;\;z \mapsto Sz =: u 
    \]       
is well-defined, linear, and continuous. Since $U$ is continuously embedded into
$C_0(\Om)$, we can consider the control-to-state map as 
    \[
        \boldsymbol{S} \coloneqq E \circ S : Z \rightarrow  C_0(\Om), 
    \]
where $E$ is the appropriate embedding operator. This allows us to define the admissible control set as
    \[
        \widehat{Z}_{ad} := \left\{ z \in Z \; : \; 
            z \in Z_{ad}, \ \boldsymbol{S} z \in \mathcal{K} \right\} ,
    \]
and as a result, the reduced minimization problem is given by 
     \begin{equation}\label{eq:rpDir}
        \min_{z \in \widehat{Z}_{ad}} \mathcal{J}(z) := J( \boldsymbol{S} z,z)=\frac{1}{2}\|\boldsymbol{S}z-u_d\|_{L^2{(\Om)}}^2 + \frac{\alpha}{2} \|z\|^2_{L^2(\Om)} . 
    \end{equation}
    
Existence and uniqueness of $z \in \widehat{Z}_{ad}$ that solves \eqref{eq:rpDir} is discussed in 
\cite[Theorem~4.1]{HAntil_DVerma_MWarma_2019b}. Next, we will state the first 
order necessary and sufficient optimality conditions after making the following assumption:
    \begin{assumption}[\bf Slater condition]\label{ass:data_compatibility}
        There is some control function $\widehat{z} \in Z_{ad}$ 
        such that the corresponding state $u$ solving the state equation
        \eqref{eq:Sd} fulfills the strict state constraint 
    \[
        u(x) < u_b(x) \quad \forall x \in \overline\Omega . 
    \]
    \end{assumption}   
Given the optimal solution $(\bar{u},\bar{z})$ to \eqref{eq:Sdd}, from \cite[Theorem~4.3]{HAntil_DVerma_MWarma_2019b} we know that there exists a measure  $\bar{\mu} \in \mathcal{M}(\Om)$ and an adjoint variable $\bar{\xi} \in L^{p'}(\Om)$ satisfying 
\begin{subequations}
	\begin{align}
		&(-\Delta)^s_D \bar{u} = \bar{z}   \quad && \mbox{in } \Om , \label{eq:a}  \\             
		&\langle \bar{\xi},(-\Delta)^s_D v \rangle_{L^{p'}(\Om),L^{p}(\Om)} 
        = \left( \bar{u}-u_d,v \right)_{L^2(\Om)} + \int_\Om v \; d\bar{\mu}    
         && \forall \;v \in U, \label{eq:b} \\
        &\langle \bar{\xi} + \alpha \bar{z} , 
        z - \bar{z} \rangle_{L^{p'}(\Om),L^{p}(\Om)} 
         \ge 0  && \forall \;z \in Z_{ad},  \label{eq:c}  \\
        &\bar{\mu} \ge 0, \quad  \bar{u}(x) \le u_b(x) \mbox{ in } \Om, 
        \quad \mbox{and} \quad \int_{\Om} (u_b - \bar{u})\;d\bar\mu = 0 \label{eq:d}.      
	\end{align}
\end{subequations}

We emphasize that it is possible to establish a Sobolev regularity of the adjoint 
variable $\bar{\xi}$, we refer to \cite[Corollary~6.6]{HAntil_DVerma_MWarma_2019b} for 
the details.
For notational convenience, from hereon, we will assume that $u_b\equiv0$. However,
the entire discussion, under minor modifications, holds without this assumption.

\section{Moreau-Yosida Regularized Optimal Control Problem}
\label{s:rocp}

The purpose of this section is to introduce a regularized optimal control problem. In particular, we use the well-known Moreau-Yosida regularization. The resulting regularized optimal control problem is given by
\begin{subequations}\label{eq:regdcp}
	\begin{equation}\label{eq:regJd}
	\min J^\ga(u,z):=J(u,z)+ \frac{1}{2 \ga} \|(\hat{\mu}+\ga u)_+\|^2_{L^2(\Om)},
	\end{equation}
	subject to the fractional elliptic PDE: Find $u \in U$ solving 
	\begin{equation}
	\begin{cases}
	(-\Delta)^s u &= z \quad \mbox{in } \Om, \\
	u &= 0 \quad \mbox{in } \RR^N\setminus \Om  ,
	\end{cases}                
	\end{equation} 
	with 
\begin{equation}\label{eq:regZd}
z \in Z_{ad}, 
\end{equation}
\end{subequations}
where $\ga >0$ denotes the regularization parameter and $0\leq \hat{\mu}\in L^2(\Om)$ is 
the realization of the Lagrange multiplier $\bar{\mu}$, see \cite{KIto_KKunisch_2008a} 
for a discussion on this matter.

Again using the control-to-state mapping, \eqref{eq:regdcp} can be reduced to (for 
notation simplicity, we again use $\boldsymbol{S}$ to denote the control-to-state map)
\begin{equation}\label{eq:regrpDir}
	\min_{z \in Z_{ad}} \mathcal{J}^{\ga}(z) :=\mathcal{J}(z)+\frac{1}{2 \ga} \|(\hat{\mu}+\ga \boldsymbol{S}z)_+\|^2_{L^2(\Om)} . 
\end{equation}
Existence and uniqueness of solution to \eqref{eq:regdcp} can be done using the standard
direct method, see for instance \cite[Theorem~4.1]{HAntil_DVerma_MWarma_2019b}. 

Moreover, we have the following first order necessary and sufficient optimality conditions
whose proof is similar to \cite[Theorem~4.2]{HAntil_DVerma_MWarma_2019b} and has been omitted for brevity.  
\begin{theorem}\label{thm:optcond}
	Let $J^\ga : L^2(\Om) \times Z \rightarrow \mathbb{R}$ be continuously Fr\'echet differentiable, and $(\bar{u}^\ga,\bar{z}^\ga) \in \widetilde{W}^{s,2}_0(\Om)\times Z_{ad}$ be a solution to the optimization
	problem \eqref{eq:regdcp}. Then there exists a Lagrange multiplier (adjoint variable) $\bar{\xi}^\ga \in 
	\widetilde{W}^{s,2}_0(\Om)$ 
	such that
	\begin{subequations}
		\begin{align}
		&\mathcal{E} (\bar{u}^\ga,v) = (\bar{z}^\ga, v)_{L^2(\Om)}  \quad && \forall \;v \in \widetilde{W}^{s,2}_0(\Om) , \label{eq:rega}  \\             
		&\mathcal{E}(\bar{\xi}^\ga,v) 
		= (\bar{u}^\ga-u_d, v)_{L^2(\Om)} + ((\hat{\mu}+\ga \bar{u}^\ga)_+,  v )_{L^2(\Om)}     
		  &&\forall \;v \in \widetilde{W}^{s,2}_0(\Om), \label{eq:regb} \\
		&(\bar{\xi}^\ga + \alpha \bar{z}^\ga, z-\bar{z}^\ga)_{L^2(\Om)}
		\ge 0  \quad && \forall \;z \in Z_{ad}  \label{eq:regc}.  
		\end{align}
	\end{subequations} 
\end{theorem}

\section{Convergence analysis}
\label{s:convergence}
The purpose of this section is to show that the regularized problem \eqref{eq:regdcp} is indeed an approximation to the original problem \eqref{eq:dcp}.
We follow the approach of \cite{HiHi2009} and adapt it to the fractional case. 
For completeness, and to make the paper self-contained, we provide almost all the details.  
We begin by obtaining a uniform bound on the regularization term. Recalling the definition 
\eqref{eq:regrpDir}, observe that for $\ga\ge 1$,
\begin{align} \label{eq:rel}
\begin{aligned}
	\mathcal{J}(\bar{z}^\ga)\le \mathcal{J}^\ga(\bar{z}^\ga)\le \mathcal{J}^\ga(\bar{z})
	&\le \mathcal{J}(\bar{z})+\frac{1}{2 \ga} \|\hat{\mu}\|^2_{L^2(\Om)} \\
	&\le \mathcal{J}(\bar{z})+\frac{1}{2} \|\hat{\mu}\|^2_{L^2(\Om)} \eqqcolon C_{\bar{z}} ,
\end{aligned}	
\end{align}
where in the second inequality we have used that $\bar{z}^\ga$ is the minimizer of \eqref{eq:regdcp}. Moreover, the second to last inequality in \eqref{eq:rel} follows from the fact that $\bar{u} \le 0$ (since $u_b = 0$) and $\hat{\mu}\ge 0$ a.e. in $\Om$.

Using \eqref{eq:rel} it follows that $\frac{1}{2 \ga} \|(\hat{\mu}+\ga \bar{u}^\gamma)_+\|^2$ is uniformly bounded. Also from \eqref{eq:rel} we have that 
\begin{align*}
	\|(\bar{u}^\ga)_+\|^2_{L^2(\Om)}
	&=\frac{1}{\ga^2}\|(\ga \bar{u}^\ga)_+\|^2_{L^2(\Om)}  \\
	  &\le \frac{1}{\ga^2} \| (\hat{\mu} + \ga \bar{u}^\ga)_+\|^2_{L^2(\Om)}=\frac{2}{\ga} \left(\mathcal{J}^\ga(\bar{z}^\ga)-\mathcal{J}(\bar{z}^\ga)\right)\\
	&\le \frac{2}{\ga} \left(\mathcal{J}(\bar{z})-\mathcal{J}(\bar{z}^\ga)+\frac{1}{2\ga} \|\hat{\mu}\|^2_{L^2(\Om)}\right) .
\end{align*}
Thus 
\begin{align} \label{uplusbnd}
	\|(\bar{u}^\ga)_+\|_{L^2(\Om)}\le \om(\ga^{-1})\ga^{-\frac{1}{2}},
\end{align}
where 
\[ 
	\om(\ga^{-1}) := 2\text{ max }\left(\frac{1}{2\ga} \|\hat{\mu}\|^2_{L^2(\Om)},\left(\mathcal{J}(\bar{z})-		\mathcal{J}(\bar{z}^\ga)\right)_+ \right)^{1/2}.
\]
Since $\bar{z}^\ga\rightarrow \bar{z}$ strongly in $L^2(\Om)$ by \cite[Proposition~2.1]{hintermuller_kunisch} and $\mathcal{J}$ is continuous, we obtain that $\om (\gamma^{-1})\downarrow0$ as $\gamma \to \infty$.
Moreover, from \eqref{eq:rel} we have $\mathcal{J}(\bar{z}^\ga)\le C_{\bar{z}}$ which yields
\begin{align} \label{regsolbound}
	\text{max}\left( \|\bar{u}^\ga - u_d\|^2_{L^2(\Om)}, \alpha \|\bar{z}^\ga\|^2_{L^2(\Om)}  \right) \le 2 C_{\bar{z}}. 
\end{align}
Then from \eqref{eq:regb}, in conjunction with \eqref{uplusbnd} and \eqref{regsolbound},
we obtain that  
\begin{align*} 
	\|\bar{\xi}^\ga\|_{\widetilde{W}^{s,p}_0(\Om)}
	& \le C\left(\|\bar{u}^\ga - u_d\|_{L^2(\Om)} + \|\hat{\mu}\|_{L^2(\Om)}+ \gamma\|(\bar{u}^\ga)_+\|_{L^2(\Om)} \right) \\
	& \le C\left(2+\om(\ga^{-1})\sqrt{\ga}) \right),
\end{align*} 
where the generic constant $C$ is independent of $\ga$.

We now estimate the distance between $(\bar{u},\bar{z})$ and $(\bar{u}^{\ga},\bar{z}^\ga)$. In the following proof and throughout the rest of the paper, all integrals will be assumed to be Lebesgue, unless otherwise noted.
\begin{theorem} \label{distance}
Let	$(\bar{u},\bar{z})$ and $(\bar{u}^{\ga},\bar{z}^\ga)$ denote the solutions of \eqref{eq:dcp} and \eqref{eq:regdcp} respectively. Then,
\begin{align} \label{eq:distance}
	\alpha \|\bar{z}-\bar{z}^\ga\|^2_{L^2(\Om)} &+ \|\bar{u}-\bar{u}^\ga\|^2_{L^2(\Om)}+\ga \|(\bar{u}^\ga)_+\|^2_{L^2(\Om)} \nonumber \\
	&\le \frac{1}{\ga} \|\hat{\mu}\|^2_{L^2(\Om)}+
	\left\langle \bar{\mu},\bar{u}^\ga \right\rangle _{\mathcal{M}(\Om),C_0(\Om)},
\end{align}
and hence
\begin{align} \label{eq:violation}
\|(\bar{u}^\ga)_+\|_{L^2(\Om)}\le \sqrt{\frac{2}{\ga}} \max \left( \frac{1}{\ga}\|\hat{\mu}\|^2_{L^2(\Om)},\left\langle \bar{\mu},(\bar{u}^\ga)_+ \right\rangle _{\mathcal{M}(\Om),C_0(\Om)}\right)^{1/2}.
\end{align}
\end{theorem}

\begin{proof}
	We start with the use of optimality conditions \eqref{eq:c} with $z=\bar{z}^\ga$ and $z=\bar{z}$ in \eqref{eq:regc}.  This  yields
	\begin{align}
		\alpha \|\bar{z}-\bar{z}^\ga\|^2_{L^2(\Om)} 
		 &\le \int_{\Om} (\bar{z}-\bar{z}^\ga)(\bar{\xi}^\ga - \bar{\xi}) 
		 . \label{eq:4a}
	\end{align}		 
	Next we write the state equations \eqref{eq:Sd} and \eqref{eq:rega} as
	\[
		\langle (-\Delta)^s_D\bar{u}, v \rangle_{L^p(\Om),L^{p'}(\Om)} 
		= \langle \bar{z}, v \rangle_{L^p(\Om),L^{p'}(\Om)}  \quad \forall v \in L^{p'}(\Om), 
	\]
	and 
	\[
		\langle (-\Delta)^s_D\bar{u}^\ga, v \rangle_{L^p(\Om),L^{p'}(\Om)} 
		= \langle \bar{z}^\ga, v \rangle_{L^p(\Om),L^{p'}(\Om)}  \quad \forall v \in L^{p'}(\Om). 
	\]
	By subtracting them, we obtain that 
	\begin{equation}\label{eq:4b}
		\langle (-\Delta)^s_D(\bar{u} - \bar{u}^\ga), v \rangle_{L^p(\Om),L^{p'}(\Om)} 
		= \langle \bar{z} - \bar{z}^\ga, v \rangle_{L^p(\Om),L^{p'}(\Om)}  \quad \forall v \in L^{p'}(\Om).
	\end{equation}
	Recall that both $\bar{\xi} \in L^{p'}(\Om)$ and $\bar{\xi}^\ga \in \widetilde{W}^{s,2}_0(\Om) \hookrightarrow L^{p'}(\Om)$ (see Remark~\ref{remark}), so we can set 
	$v := \bar{\xi}^\ga - \bar{\xi}$ in \eqref{eq:4b}. Subsequently, substituting \eqref{eq:4b} in \eqref{eq:4a}, we obtain 
	\begin{equation}\label{eq:4c}
	\begin{aligned}
		\alpha \|\bar{z}-\bar{z}^\ga\|^2_{L^2(\Om)} 
		&\le \langle (-\Delta)^s_D(\bar{u} - \bar{u}^\ga), \bar{\xi}^\ga - \bar{\xi} \rangle_{L^p(\Om),L^{p'}(\Om)} \\
		&= -\|\bar{u} - \bar{u}^\ga \|_{L^2(\Om)}^2 +\int_{\Om} (\hat{\mu}+\ga \bar{u}^\ga)_+(\bar{u}-\bar{u}^\ga)\\
		&\quad - \left\langle \bar{\mu},\bar{u}-\bar{u}^\ga \right\rangle _{\mathcal{M}(\Om),C_0(\Om)}
	\end{aligned}	
	\end{equation}
	where we have used \eqref{eq:b} and \eqref{eq:regb}, and that 
	$\bar{u}, \bar{u}^\gamma \in U$.
	Moreover, from \eqref{eq:d} we have that $\left\langle \bar{\mu},\bar{u} \right\rangle _{\mathcal{M}(\Om),C_0(\Om)} = 0$. Substituting this in \eqref{eq:4c}, we obtain 
	that 
	\begin{equation}\label{eq:4d}
	\begin{aligned}
		\alpha \|\bar{z}-\bar{z}^\ga\|^2_{L^2(\Om)} 	
		+ \|\bar{u} - \bar{u}^\ga \|_{L^2(\Om)}^2 	
		\hspace{-2pt} &= \hspace{-1pt}\left\langle \bar{\mu},\bar{u}^\ga \right\rangle _{\mathcal{M}(\Om),C_0(\Om)} \hspace{-1pt}+ \hspace{-1pt}
		\int_{\Om} (\hat{\mu}+\ga \bar{u}^\ga)_+(\bar{u}-\bar{u}^\ga)
		  \\
		&\le \left\langle \bar{\mu},\bar{u}^\ga \right\rangle _{\mathcal{M}(\Om),C_0(\Om)} +
		\int_{\Om} (\hat{\mu}+\ga \bar{u}^\ga)_+(-\bar{u}^\ga)		    
	\end{aligned}	
	\end{equation}
	where we have used that $\int_{\Om} (\hat{\mu}+\ga \bar{u}^\ga)_+(\bar{u}) \le 0$.
	Next, we shall analyze the integral $\int_{\Om} (\hat{\mu}+\ga \bar{u}^\ga)_+(-\bar{u}^\ga)$. 
	Let 
	\begin{equation}\label{eq:4f}
		\Om_{\ga}^+(\hat{\mu})\coloneqq \{x\in \Om \; : \;\hat{\mu}+\ga \bar{u}^\ga >0\},
	\end{equation}
	then we have that $ \int_{\Om \setminus \Om_{\ga}^+(\hat{\mu})} 
		  (\hat{\mu}+\ga \bar{u}^\ga)_+(-\bar{u}^\ga) = 0.$ This allows us to write
	\begin{align}\label{eq:4e}
			\int_{\Om} (\hat{\mu}+\ga \bar{u}^\ga)_+(-\bar{u}^\ga)
		&= \int_{\Om_{\ga}^+(\hat{\mu})} (\hat{\mu}+\ga \bar{u}^\ga)(-\bar{u}^\ga) \nonumber \\
		&\le - \ga \|\bar{u}^\ga\|^2_{L^2(\Om_{\ga}^+(\hat{\mu}))}  
		   -\int_{\Om_{\ga}^+(\hat{\mu})}\hat{\mu} \bar{u}^\ga	
		   \nonumber \\
		&\le - \ga \|(\bar{u}^\ga)_+\|^2_{L^2(\Om)}  
		   + \frac{1}{\ga} \|\hat{\mu}\|^2_{L^2(\Om)},	 	 
	\end{align}
	where in the first term in the last step we have used that 
\begin{equation} \label{eq:uPlusNorm}
\ga \|\bar{u}^\ga\|^2_{L^2(\Om_{\ga}^+(\hat{\mu}))} \ge \ga \|\bar{u}^\ga\|^2_{L^2(\Om_{\ga}^+(0))}=\ga \|(\bar{u}^\ga)_+\|^2_{L^2(\Om)},
\end{equation}
because $\Om_{\ga}^+(\hat{\mu})\supseteq \Om_{\ga}^+(0):=\{x \in \Omega : \bar{u}^\ga>0\}$.
	Moreover, we have estimated the second term by using the fact that 
	$\bar{u}^\ga> \dfrac{-\hat{\mu}}{\ga}$ in $\Om_{\ga}^+(\hat{\mu})$.
	Substituting \eqref{eq:4e} in \eqref{eq:4d}, we obtain \eqref{eq:distance}. 
	
	Next, we have that $(\bar{u}^\ga)_+\in C_0(\Om)$ because $\bar{u}^\ga$ being the solution to state equation \eqref{eq:Sd} is in $C_0(\Om)$. This together with non negativity of $\bar{\mu}\in \mathcal{M}(\Om)$ combined with \eqref{eq:distance}, yields \eqref{eq:violation}.
\end{proof}

\begin{theorem} \label{ydorder}
	If $0\in Z_{ad}$ and $u_d\ge 0$ a.e. in $\Om$, then $\|(\bar{u}^\ga)_+\|_{L^2(\Om)}=\mathcal{O} (\ga^{-1})$ as $\ga \rightarrow \infty$.
\end{theorem}

\begin{proof}
	Recall that the regularized state and the adjoint variables $\bar{u}^\ga$ and $\bar{\xi}^\ga$ satisfy
	\begin{align*}
		\mathcal{E}(\bar{u}^\ga,v)&=\int_{\Om} \bar{z}^\ga v ,\; && \forall v\in \widetilde{W}^{s,2}_0(\Om),\\
		\mathcal{E}(\bar{\xi}^\ga,v)&=\int_{\Om} (\bar{u}^\ga - u_d)v + \int_{\Om} (\hat{\mu}+\ga \bar{u}^\ga)_+ v ,\; && \forall v\in \widetilde{W}^{s,2}_0(\Om).
	\end{align*}
	Substituting $v=\bar{\xi}^\ga$ in the first equation and $v=\bar{u}^\ga$ in the second and then subtracting yields
	\begin{equation}\label{eq:5a}
		0=\|\bar{u}^\ga\|^2_{L^2(\Om)}-\int_{\Om} \bar{u}^\ga u_d +\int_{\Om} (\hat{\mu} +\ga\bar{u}^\ga)_+ \bar{u}^\ga - \int_{\Om} \bar{z}^\ga \bar{\xi}^\ga.
	\end{equation}
	Using the definition of the set $\Om_\ga^+(\hat\mu)$ from \eqref{eq:4f} in conjunction with \eqref{eq:5a}, we obtain that 
\begin{alignat*}{5}
		0&\ge \|\bar{u}^\ga\|^2_{L^2(\Om)}-\int_{\Om} \bar{u}^\ga u_d +\int_{\Om_{\ga}^+(\hat{\mu})} \hat{\mu} \bar{u}^\ga + \int_{\Om_{\ga}^+(\hat{\mu})} \ga (\bar{u}^\ga)^2 - \int_{\Om} \bar{z}^\ga \bar{\xi}^\ga\\
	  &\ge \|\bar{u}^\ga\|^2_{L^2(\Om)}-\int_{\Om} \bar{u}^\ga u_d -\int_{\Om_{\ga}^+(\hat{\mu})} \ga^{-1} \hat{\mu}^2  +\int_{\Om_{\ga}^+(0)} \ga (\bar{u}^\ga)^2 - \int_{\Om} \bar{z}^\ga \bar{\xi}^\ga,
\end{alignat*}
where in the last inequality we have used a similar technique as in the end of the previous proof. Setting  $z=\frac{\bar{z}^\ga}{2} \in Z_{ad}$ (because $Z_{ad}$ is convex and $0 \in Z_{ad}$) in \eqref{eq:regc}, we obtain that 
	\begin{alignat*}{3}
		0  &\ge \|\bar{u}^\ga\|^2_{L^2(\Om)}-\int_{\Om} \bar{u}^\ga u_d -\int_{\Om_{\ga}^+(\hat{\mu})} \ga^{-1} \hat{\mu}^2  +\int_{\Om_{\ga}^+(0)} \ga (\bar{u}^\ga)^2 +\alpha \int_{\Om} (\bar{z}^\ga)^2\\
		&\geq   \|\bar{u}^\ga\|^2_{L^2(\Om)}-\int_{\Om} \bar{u}^\ga u_d -\int_{\Om_{\ga}^+(\hat{\mu})} \ga^{-1} \hat{\mu}^2  + \ga \|(\bar{u}^\ga)_+\|_{L^2(\Omega)}^2 +\alpha \|\bar{z}^\ga\|_{L^2(\Omega)}^2,
\end{alignat*}
where we have once again used \eqref{eq:uPlusNorm}.  From here, we can bound
	\[
		\ga \|(\bar{u}^\ga)_+\|_{L^2(\Om)}^2\le \int_{\Om_{\ga}^+(\hat{\mu})} \ga^{-1} \hat{\mu}^2 +\int_{\Om} \bar{u}^\ga u_d.
	\]
	Now since $u_d\ge 0$ a.e. in $\Om$ 
	\[
		\ga \|(\bar{u}^\ga)_+\|_{L^2(\Om)}^2\le \int_{\Om} \ga^{-1} \hat{\mu}^2 +\int_{\Om} (\bar{u}^\ga)_+ u_d. 
	\]
	We can now apply the Cauchy-Schwarz inequality and Young's inequality to the second term on the right-hand side to obtain
\[
		\ga \|(\bar{u}^\ga)_+\|_{L^2(\Om)}^2\le \int_{\Om} \ga^{-1} \hat{\mu}^2 +\frac{\gamma}{2} \|(\bar{u}^\ga)_+\|_{L^2(\Omega}^2 + \frac{1}{2\gamma} \| u_d\|_{L^2(\Omega)}^2, 
	\]	
	from which the result follows.
\end{proof}
Now, we will relax the condition on $u_d$ in Theorem \ref{ydorder}.

\begin{theorem} \label{relygaorder}
	If there exists $\epsilon>0$ such that 
	\[ 
		-\int_{\Om} u_d^- \bar{u} -\|\bar{u}\|^2_{L^2(\Om)} -\alpha \|\bar{z}\|^2_{L^2(\Om)} \le -\epsilon,
	\]
	then $\|(\bar{u}^\ga)_+\|_{L^2(\Om)}=\mathcal{O} (\ga^{-1})$ as $\ga \rightarrow \infty$. Here, $u_d^-$ denotes the negative part of the function $u_d$. 
\end{theorem}

\begin{proof}
	It is sufficient to show that $\bar{z}^\ga\rightarrow \bar{z}$ in $L^2(\Om)$ and $\bar{u}^\ga\rightarrow \bar{u}$ in $\widetilde{W}^{s,2}_0(\Om)$. Then we can argue exactly
	as in \cite[Theorem~2.3]{HiHi2009}. 
\end{proof}

\section{Finite element discretization and convergence analysis}
\label{s:fem}

In this section, we will discuss an a priori analysis of error due to discretization. 
We assume we have a quasi-uniform triangulation of $\Omega$, denoted $\mathcal T_h$, consisting of triangles $T$ such that $\cup_{T \in \mathcal{T}_h}={\overline{\Om}}$. Given $\mathcal T_h$, we define the finite element space for the state and adjoint variables as
\[ 
 U_h=\{ u_h\in C_0(\Om): u_h\big|_T \in \mathcal{P}_1,\;\forall T\in \mathcal{T}_h\},
\]
where $\mathcal P_1$ is the space of polynomials of degree at most one. 

Next we define the admissible discrete control space as 
\[
Z_{ad,h} := \{z_h \in Z_{ad} : z_h\big|_T \in \mathcal P_0, \ \forall T \in \mathcal T_h\},
\]
i.e., the admissible controls are piecewise constant on the triangulation $\mathcal T_h$.

For a given $z_h \in Z_{ad,h}$, there exists a unique solution to the the problem find $u_h\in U_h$ such that 
\[
	\mathcal{E}(u_h,v_h)=\int_{\Om}z_h v_h,\;\;\; \forall v_h\in U_h,
\]
where $\mathcal{E}$ is as defined in \eqref{intbyParts}. We can therefore define the discrete solution operator $\boldsymbol{S}_h:Z_{ad}\rightarrow U_h$ whose action is given by $z_h \mapsto \boldsymbol{S}_h z_h =:u_h$. 
Now the fully discrete version of the regularized problem \eqref{eq:regdcp} in reduced form is given by 	
	\begin{equation} \label{eq:fullDisc}
		\min_{z_h\in Z_{ad,h}} \mathcal J_h^\ga(z_h):= \frac12 \|\boldsymbol S_h z_h - u_d\|_{L^2(\Omega)}^2 + \frac{\alpha}{2} \|z_h\|_{L^2(\Omega)}^2 + \frac{1}{2 \ga} \|(\hat{\mu}+\ga \boldsymbol{S}_h z_h)_+\|^2_{L^2(\Om)}.
	\end{equation}

The following result is a discrete analogue of Theorem \ref{thm:optcond} stating the first order optimality conditions for \eqref{eq:fullDisc}. 
\begin{theorem}\label{thm:discoptcond} 
	Let $\mathcal J_h^\ga :L^p(\Om) \rightarrow \mathbb{R}$ be continuously Fr\'echet differentiable, and $(\bar{u}_h^\ga,\bar{z}_h^\ga)$ be a solution to the optimization problem \eqref{eq:fullDisc}. There exists a Lagrange multiplier (adjoint variable) $\bar{\xi}_h^\ga \in U_h$ such that
		\begin{align*}
		 &\mathcal{E}(\bar{u}_h^\ga,v_h) = (\bar{z}_h^\ga,v_h)_{L^2(\Om)}  \quad &&\forall v_h \in U_h ,\\             
		&\mathcal{E}(\bar{\xi}_h^\ga,v_h) 
		= \left( \bar{u}_h^\ga-u_d , v_h \right)_{L^2(\Om)} 
		  +\left( (\hat{\mu}+\ga \bar{u}_h^\ga)_+ , v_h \right)_{L^2(\Om)}     
		 && \forall \;v_h \in U_h,\\
		&\big( \bar{\xi}_h^\ga + \alpha \bar{z}_h^\ga , z_h-\bar{z}_h^\ga
		\big)_{L^2(\Om)} \ge 0  && \forall \;z_h \in Z_{ad,h}.  
		\end{align*} 
\end{theorem}

In order to discuss the convergence of our discrete approximations to the actual solution, we assume that we have a sequence of meshes $\mathcal T_h$ and corresponding solutions spaces $U_h$ and $Z_{ad,h}$ that provide better approximations as $h \to 0$. The quantity we wish to study is $\|\bar{z} - \bar{z}_h^\gamma\|_{L^2(\Omega)}$, which we will split as 
\begin{equation} \label{eq:split}
\|\bar{z} - \bar{z}_h^\gamma\|_{L^2(\Omega)} \leq \|\bar{z} - \bar{z}^\gamma\|_{L^2(\Omega)} + \|\bar{z}^\gamma - \bar{z}_h^\gamma\|_{L^2(\Omega)}, 
\end{equation}
where $\bar{z}$ solves \eqref{eq:rpDir}, $\bar{z}^\gamma$ solves \eqref{eq:regrpDir}, and $\bar{z}_h^\gamma$ solves \eqref{eq:fullDisc}.  
The first term on the right-hand side in \eqref{eq:split} can be estimated by the decay of $(\bar{u}^\gamma)_+$ as seen in Theorem \ref{distance}.  As such, we present the following lemma regarding the approximation of $(\bar{u}^\gamma)_+$. 

\begin{lemma} \label{lem:uPlus}
Let $\bar{u}^\gamma = \boldsymbol{S}\bar{z}^\gamma$. 
Then we have 
\begin{equation}\label{eq:uGamPlusBound}
\|(\bar{u}^\gamma)_+\|_{C_0(\Om)} \leq C \big( \| (\bar{u}^\gamma)_+ - I_h(\bar{u}^\gamma)_+ \|_{C_0(\Om)} + \gamma^{-1/2} h^{-N/2}\big),
\end{equation}
where the constant is independent of $\gamma$ and $h$.
\end{lemma}
\begin{proof}
Let $I_h$ be the Lagrange interpolant, then using the inverse estimates (for example \cite{ErGu2004}) we have that
\begin{alignat*}{5}
\|(\bar{u}^\gamma)_+\|_{C_0(\Omega)} &\leq \|(\bar{u}^\gamma)_+ - I_h(\bar{u}^\gamma)_+\|_{C_0({\Omega})} + \|I_h(\bar{u}^\gamma)_+\|_{C_0({\Omega})}\\
& \leq C\big(\|(\bar{u}^\gamma)_+ - I_h(\bar{u}^\gamma)_+\|_{C_0({\Omega})} + h^{-N/2}\|(\bar{u}^\gamma)_+\|_{L^2(\Omega)}\big).
\end{alignat*}
The result follows by using \eqref{uplusbnd}.
\end{proof}

\begin{remark}
{\rm Notice that since $\bar{u}^\gamma \in C_0(\Om)$, therefore the interpolation error 
$\|(\bar{u}^\gamma)_+ - I_h(\bar{u}^\gamma)_+\|_{C_0(\Om)} \rightarrow 0$ as $h \rightarrow 0$. 
In the case that $\bar{u}^\gamma \in \widetilde{W}^{s,q}(\Om)$ for a sufficiently large $q$, for instance, 
$q > \frac{N}{s}$ so that $\widetilde{W}^{s,q}(\Om) \hookrightarrow C_0(\Om)$, we can 
bound the interpolation error in \eqref{eq:uGamPlusBound} as 
$\|(\bar{u}^\gamma)_+ - I_h(\bar{u}^\gamma)_+\|_{C_0(\Om)} \le C h^{s-\frac{N}{q}}$.
}
\end{remark}

Before we obtain our first result for the convergence of $\|\bar{z}^\gamma - \bar{z}_h^\gamma\|_{L^2(\Omega)}$, we introduce some notation to make the computations in what follows more tractable.  We define the quantity
\[
f(u) := u - u_d + (\hat{\mu} + \gamma u)_+, 
\]
so that the adjoint equation \eqref{eq:regb} becomes
\[
\mathcal{E} (\bar{\xi}^\gamma, v) = (f(\bar{u}^\gamma), v)_{L^2(\Om)}, \qquad \forall v \in \widetilde{W}^{s,2}_0(\Om). 
\]
Furthermore, we introduce the continuous and discrete adjoint problem operators as $R := \boldsymbol{S}^*$ and $R_h := \boldsymbol{S}_h^*$, and note that these operators are well defined, linear, and bounded.  Finally, we define $\Pi_h : L^2(\Omega) \to Z_{ad,h}$, to be the $L^2$-orthogonal projection.  With this we are ready to state our convergence results. In subsection~\ref{s:gdep} we shall first state the results for all $s$ but where the constants in convergence can be $\gamma$-dependent. In subsection~\ref{s:gindep}, we state the results for a range of $s$ where these constants are $\gamma$-independent.
%

\subsection{$\gamma$-Dependent Convergence and Error Bounds}
\label{s:gdep}

\begin{theorem}\label{thm:errBound1}
Let $\bar{z}^\ga$ and $\bar{z}_h^\ga$ denote the solution of \eqref{eq:regrpDir} and \eqref{eq:fullDisc} respectively with corresponding states $\bar{u}^\ga = \boldsymbol{S}\bar{z}^\gamma$ and $\bar{u}_h^\ga = \boldsymbol{S}_h \bar{z}^\gamma$.   Then the discretization error can be bounded as 
\begin{alignat}{4}
\|\bar{z}^\gamma - \bar{z}_h^\gamma\|_{L^2(\Omega)} \leq C  \frac{1 + \gamma}{\alpha}\Big(& \|(R - R_h)f(\boldsymbol S\bar{z}^\gamma)\|_{L^2(\Omega)} + \|(\boldsymbol S - \boldsymbol S_h) \bar{z}^\gamma\|_{L^2(\Omega)} \nonumber \\
& + \|\Pi_h \bar{\xi}^\gamma - \bar{\xi}^\gamma \|_{L^2(\Omega)} + \|\Pi_h \bar{z}^\gamma - \bar{z}^\gamma\|_{L^2(\Omega)}\Big),	\label{eq:errBound}
\end{alignat}
where $C$ is a positive constant independent of $\gamma$ and $h$.
\end{theorem}
\begin{proof}
	In what follows all of the norms and inner products used are in $L^2(\Omega)$ unless noted otherwise.  Using the notation introduced above, we begin by defining the auxiliary variables
\begin{alignat*}{5}
\hat{u}_h^\ga  &:= \boldsymbol{S}_h \bar{z}^\gamma,\\
\hat{\xi}_h^\ga &: = R_h f(\bar{u}^\gamma) = R_h f(\boldsymbol{S}\bar{z}^\gamma),\\
\widetilde{\xi}_h^\ga &: = R_h f(\hat{u}_h^\gamma) = R_h f(\boldsymbol{S}_h \bar{z}^\gamma). 
\end{alignat*}
Next, we recall the continuous and discrete optimality conditions
\begin{alignat*}{3}
 (\bar{\xi}^\ga + \alpha \bar{z}^\ga, z-\bar{z}^\ga)
		 &\ge 0 \qquad \forall \;z \in Z_{ad},\\ 
(\bar{\xi}_h^\ga +\alpha \bar{z}_h^\ga, z_h-\bar{z}_h^\ga)
		&\ge 0 \qquad \forall \;z_h \in Z_{ad,h}.
\end{alignat*}
We then perform the following: (i) replace $z$ with $\bar{z}_h^\gamma$ in the first inequality and $z_h$ with $\Pi_h \bar{z}^\ga$ in the second inequality; (ii) add the two inequalities; (iii) introduce and rearrange terms appropriately to obtain
\begin{alignat*}{4}
\alpha \| \bar{z}^\gamma - \bar{z}_h^\gamma\|^2 &\leq (\bar{\xi}^\gamma - \bar{\xi}_h^\gamma, \bar{z}_h^\gamma - \bar{z}^\gamma)  &&+ (\bar{\xi}_h^\gamma + \alpha \bar{z}_h^\gamma, \Pi_h \bar{z}^\gamma - \bar{z}^\gamma)\\
& = (a) && + (b). 
\end{alignat*}
We can further separate ($a$) into
\begin{alignat}{5}
(\bar{\xi}^\gamma - \bar{\xi}_h^\gamma, \bar{z}_h^\gamma - \bar{z}^\gamma) &=  (\bar{\xi}^\gamma - \hat{\xi}_h^\gamma + \hat{\xi}_h^\gamma - \widetilde{\xi}_h^\gamma + \widetilde{\xi}_h^\gamma -  \bar{\xi}_h^\gamma, \bar{z}_h^\gamma - \bar{z}^\gamma)\nonumber\\
& =  (\bar{\xi}^\gamma - \hat{\xi}_h^\gamma, \bar{z}_h^\gamma - \bar{z}^\gamma)  +  (\hat{\xi}^\gamma - \widetilde{\xi}_h^\gamma, \bar{z}_h^\gamma - \bar{z}^\gamma) \nonumber\\
& \qquad  + (\widetilde{\xi}^\gamma - \bar{\xi}_h^\gamma, \bar{z}_h^\gamma - \bar{z}^\gamma) \nonumber \\
& = (I) + (II) + (III). \label{eq:decomp1}
\end{alignat}
Before we look at these three terms, we note that we can bound
\begin{equation} \label{eq:fIneq}
f(u) - f(v) = u-u_d + (\hat{\mu} + \gamma u)_+ - (v- u_d + (\hat{\mu} + \gamma v)_+) \leq u - v + \gamma (u - v)_+,
\end{equation}
since $(u)_+ - (v)_+ \leq (u - v)_+$.  Now we bound 
\begin{alignat*}{5}
		(I) &=  (R f(\bar{u}^\gamma)  - R_h f(\bar{u}^\gamma),  \bar{z}_h^\gamma - \bar{z}^\gamma) \leq C \|(R - R_h) f(\boldsymbol S \bar{z}^\gamma)\| \|\bar{z}_h^\gamma - \bar{z}^\gamma\|,\\
		(II)&=  (R_h f(\bar{u}^\gamma) - R_h f(\hat{u}_h^\gamma), \bar{z}_h^\gamma - \bar{z}^\gamma)\hspace{-1.5pt} \leq \hspace{-1.5pt} C (1 + \gamma) \|R_h(\boldsymbol{S} - \boldsymbol{S}_h) \bar{z}^\gamma\|\|\bar{z}_h^\gamma - \bar{z}^\gamma\|. 
\end{alignat*}
For $(III)$, we use \eqref{eq:fIneq} to obtain
\begin{alignat*}{4}
		(III) & =  (R_h f(\hat{u}_h^\gamma) - R_h f(\bar{u}_h^\gamma), \bar{z}_h^\gamma - \bar{z}^\gamma) \\
		      & \leq  \Big(R_h (\boldsymbol{S}_h (\bar{z}^\gamma - \bar{z}_h^\gamma) + \gamma (\boldsymbol{S}_h (\bar{z}^\gamma - \bar{z}_h^\gamma))_+), \bar{z}_h^\gamma - \bar{z}^\gamma \Big) .
	\end{alignat*}
From here, we note that $(R_h\boldsymbol S_h (\bar{z}^\gamma - \bar{z}_h^\gamma), \bar{z}_h^\gamma - \bar{z}^\gamma)  \leq 0$, and so we need only consider 
\begin{alignat*}{3}
(III) &\leq  \gamma (R_h( \boldsymbol{S}_h (\bar{z}^\gamma - \bar{z}_h^\gamma)_+), \bar{z}_h^\gamma - \bar{z}^\gamma) \\
& = \gamma  ((\boldsymbol{S}_h(\bar{z}^\gamma - \bar{z}_h^\gamma))_+ , \boldsymbol{S}_h(\bar{z}_h^\gamma - \bar{z}^\gamma)) \\
& = - \gamma  ((\boldsymbol{S}_h(\bar{z}^\gamma - \bar{z}_h^\gamma))_+ ,\boldsymbol{S}_h(\bar{z}^\gamma - \bar{z}_h^\gamma)) \\
& = \begin{cases} 
- \gamma \|\boldsymbol S_h (\bar{z}^\gamma - \bar{z}_h^\gamma )\|_{L^2(\Omega)}^2 & \mbox{ if }\boldsymbol S_h \bar{z}^\gamma - \bar{z}_h^\gamma \geq 0,\\
0 & \mbox{ if } \boldsymbol S_h \bar{z}^\gamma - \bar{z}_h^\gamma < 0.
\end{cases}
\end{alignat*}
In either of the cases above, we see that $(III)$ is non positive and so it may be removed from the bound. Now, turning our attention to ($b$), we note that it can be rewritten as $( \bar{\xi}_h^\gamma, \Pi_h \bar{z}^\gamma - \bar{z}^\gamma) $ since $\Pi_h$ is an orthogonal projection.  Using similar techniques on ($b$) as we did on ($a$), we have 
\begin{alignat}{5}
(b) & =  (\bar{\xi}_h^\gamma - \widetilde{\xi}_h^\gamma + \widetilde{\xi}_h^\gamma - \hat{\xi}_h^\gamma + \hat{\xi}_h^\gamma - \bar{\xi}^\gamma + \bar{\xi}^\gamma, \Pi_h \bar{z}^\gamma - \bar{z}^\gamma) \nonumber\\ 
& =  (\bar{\xi}^\gamma - \hat{\xi}_h^\gamma, \bar{z}^\gamma - \Pi_h \bar{z}^\gamma)   +  (\hat{\xi}^\gamma - \widetilde{\xi}_h^\gamma, \bar{z}^\gamma - \Pi_h \bar{z}^\gamma) \nonumber \\
& \qquad  +  (\widetilde{\xi}^\gamma - \bar{\xi}_h^\gamma, \bar{z}^\gamma - \Pi_h\bar{z}^\gamma)   +  (\Pi_h \bar{\xi}^\gamma - \bar{\xi}^\gamma, \bar{z}^\gamma - \Pi_h\bar{z}^\gamma) \nonumber\\
& = (i) + (ii) + (iii) + (iv) \label{eq:decomp2}. 
\end{alignat} 
The quantities $(i)$ and $(ii)$ can be bounded similarly to $(I)$ and $(II)$ above with the substitution of $(\bar{z}^\gamma - \Pi_h \bar{z}^\gamma)$ for $(\bar{z}_h^\gamma - \bar{z}^\gamma)$.  Focusing on the remaining terms we obtain
\begin{alignat*}{4}
(iii) & = (R_h f(\hat{u}_h^\gamma) - R_h f (\bar{u}_h^\gamma),  \bar{z}^\gamma - \Pi_h \bar{z}^\gamma)\\
&  \leq C(1+\gamma)\|R_h \boldsymbol S_h(\bar{z}^\gamma - \bar{z}_h^\gamma)\|\|\bar{z}^\gamma - \Pi_h \bar{z}^\gamma\|\\
(iv)  & \leq \|\Pi_h \bar{\xi}^\gamma  - \bar{\xi}^\gamma\| \|\bar{z}^\gamma - \Pi_h \bar{z}^\gamma\|.
\end{alignat*}
Before combining all of the bounds, we note that the operators $\boldsymbol{S}_h$, $R_h$ and $R_h \boldsymbol{S}_h$ are uniformly bounded in $h$, that is, there exists some $\hat{C}$ independent of $h$ such that 
\[
\|\boldsymbol{S}_h\| + \|R_h\| + \|R_h \boldsymbol{S}_h\| \leq \hat{C} \qquad \forall h, 
\]
where the above norms are the appropriate operator norms.
Now combining all of our bounds we have
\begin{alignat*}{5}
\alpha \|\bar{z}^\gamma - \bar{z}_h^\gamma\|^2 \leq C\Bigg(&\|(R - R_h) f(\boldsymbol S \bar{z}^\gamma)\| \big( \|\bar{z}_h^\gamma - \bar{z}^\gamma\| +  \|\bar{z}^\gamma - \Pi_h \bar{z}^\gamma\|\big) \\
& + (1 + \gamma)\hat{C} \|(\boldsymbol{S} - \boldsymbol{S}_h) \bar{z}^\gamma\|\big(\|\bar{z}_h^\gamma - \bar{z}^\gamma\| +  \|\bar{z}^\gamma - \Pi_h \bar{z}^\gamma\|\big)\\
& + (1+\gamma) \hat{C}\|\bar{z}^\gamma - \bar{z}_h^\gamma\|\|\bar{z}^\gamma - \Pi_h \bar{z}^\gamma\|\\
& + \|\Pi_h \bar{\xi}^\gamma  - \bar{\xi}^\gamma\| \|\bar{z}^\gamma - \Pi_h \bar{z}^\gamma\|\Bigg).
\end{alignat*}
Next, we make use of Young's inequality with constant $\alpha/(2C)$ to obtain the bound 
\begin{alignat*}{3}
\alpha \|\bar{z}^\gamma - \bar{z}_h^\gamma\|^2  \leq C &\Bigg(\frac{C}{\alpha}\|(R-R_h) f(\boldsymbol{S} \bar{z}^\gamma)\|^2  + \frac{(1 + \gamma)^2\hat{C}^2C}{\alpha}\|(\boldsymbol{S} - \boldsymbol{S}_h) \bar{z}^\gamma\|^2\\ 
&  + \frac{C}{\alpha} \|\Pi_h \bar{\xi}^\gamma  - \bar{\xi}^\gamma\|^2 + \left(\frac{3\alpha}{4C} + \frac{(1+\gamma)^2 \hat{C}^2 C}{\alpha}\right) \|\Pi_h \bar{z}^\gamma - \bar{z}^\gamma\|^2 \\
& + \frac{3\alpha}{4C} \|\bar{z}_h^\gamma - \bar{z}^\gamma\|^2\Bigg).
\end{alignat*}
Grouping terms with our quantity of interest on the left and dividing by the resulting constant leads to \eqref{eq:errBound}. 
\end{proof}

We now use \eqref{eq:split} and combine  \eqref{eq:distance}, \eqref{eq:uGamPlusBound}, and \eqref{eq:errBound} to obtain the overall regularization and discretization error bound. 

\begin{corollary}
	Under the assumptions of Theorem~\ref{thm:errBound1} and for a fixed $\gamma$, we have  
	$\bar{z}^\gamma_h \rightarrow \bar{z}^\gamma$ as $h \rightarrow 0$. 
	Let 
	\[
		Z_{ad} = \{z \in Z: a \leq z \leq b \; a.e. \; in \; \Omega\}, 
	\] 
	where $a , b \in W^{s,2}(\Omega)$, with $a < b$ a.e. in $\Omega$, be given. 
	If $\Omega$ is $C^\infty$, then we get the following rate of convergence 
	\begin{alignat*}{5}
		\|\bar{z}^\gamma - \bar{z}_h^\gamma\|_{L^2(\Omega)} \leq  C \frac{1 + \gamma}{\alpha}\Big(&( h^\beta + h^s) (\|\bar{z}^\gamma\|_{L^2(\Omega)} + \|u_d\|_{L^2(\Omega)} + \|\hat{\mu}\|_{L^2(\Omega)})\\
		 &+ (h^\lambda + h^s) \|\bar{z}^\gamma\|_{W^{s,2}(\Omega)}\Big),
	\end{alignat*}
where $\beta = \min\{2s, 1 - \varepsilon\}$ and $\lambda = \min\{3s, s+ 1/2 - \varepsilon, 1 - \varepsilon\}$. 
\end{corollary}
\begin{proof}
	We first comment on the regularity of $\bar{z}^\gamma, \bar{u}^\gamma$, and $\bar{\xi}^\gamma$.  Due to $\hat{\mu}$, the right-hand side of \eqref{eq:regb} cannot be more regular than $L^2(\Omega)$, and therefore we have $\bar{\xi}^\gamma \in \widetilde{W}_0^{s,2}(\Omega)$. The optimality condition \eqref{eq:regc} is equivalent to 
	\begin{equation}\label{eq:proj_f}
		\bar{z}^\gamma = \mathbb{P}_{Z_{ad}} \left(-\alpha^{-1} \bar{\xi}^\gamma \right) 
	\end{equation}
where $\mathbb{P}_{Z_{ad}}$ is the projection onto the set $Z_{ad}$. Since $\bar{\xi}^\gamma \in \widetilde{W}_0^{s,2}(\Omega)$, therefore $\bar{\xi}^\gamma|_{\Omega} \in W^{s,2}(\Omega)$. Then using the projection formula \eqref{eq:proj_f}, in conjunction with \cite{War}, we obtain that $\bar{z}^\gamma \in W^{s,2}(\Omega)$. 

	Since $\Omega$ is assumed to have $C^\infty$ boundary, then by \cite{GGrubb_2015a} we have $\bar{u}^\gamma \in W^{\nu,2}(\Omega)$ where $\nu = \min\{2s, s + 1/2 - \varepsilon\}$.  Now for a fixed $\gamma$, we use \cite[Proposition 3.8]{BoBoNoOtSa2018} and interpolation estimates 
	to obtain
\begin{alignat*}{4}
\|\bar{z}^\gamma - \bar{z}_h^\gamma\|_{L^2(\Omega)} \leq C  \frac{1 + \gamma}{\alpha}\Big(& \|(R - R_h)f(\boldsymbol S\bar{z}^\gamma)\|_{L^2(\Omega)} + \|(\boldsymbol S - \boldsymbol S_h) \bar{z}^\gamma\|_{L^2(\Omega)} \nonumber \\
	& + \|\Pi_h \bar{\xi}^\gamma - \bar{\xi}^\gamma \|_{L^2(\Omega)} + \|\Pi_h \bar{z}^\gamma - \bar{z}^\gamma\|_{L^2(\Omega)}\Big)\\
\leq C \frac{1 + \gamma}{\alpha}\Big(& h^\beta \|f(\boldsymbol S\bar{z}^\gamma)\|_{L^2(\Omega)} + h^\lambda \|\bar{z}^\gamma\|_{W^{s,2}(\Omega)} \\
	& + h^s\|f(\boldsymbol S \bar{z}^\gamma)\|_{L^2(\Omega)} + h^s\|\bar{z}^\gamma\|_{W^{s,2}(\Omega)}\Big)\\
 \leq C \frac{1 + \gamma}{\alpha}\Big(&( h^\beta + h^s) (\|\bar{z}^\gamma\|_{L^2(\Omega)} + \|u_d\|_{L^2(\Omega)} + \|\hat{\mu}\|_{L^2(\Omega)})\\
	& + (h^\lambda + h^s) \|\bar{z}^\gamma\|_{W^{s,2}(\Omega)}\Big),
\end{alignat*}
where $\beta = \min\{2s, 1 - \varepsilon\}$, and $\lambda = \min\{3s, s+ 1/2 - \varepsilon, 1 - \varepsilon\}$. 
\end{proof}


\subsection{$\gamma$-Independent Convergence}
\label{s:gindep}

Next, we present a convergence result in $h$ which is independent of $\gamma$.
First we present a lemma that shows that we can bound the $L^1(\Omega)$ norm of $(\widehat{\mu} + \gamma \bar{u}^\gamma)_+$ and $(\widehat{\mu} + \gamma \bar{u}_h^\gamma)_+$ independent of $\gamma$.  The proof follows a similar argument as \cite[Lemma 3.4]{HiHi2009}, but we include it in full as there are some needed modifications due to the fact that, in addition to the state constraints, we also have the control constraints. Also the regularity results from the classical case ($s=1$), do not carry over to the fractional case.

\begin{lemma}\label{lem:L1bound}
Let $(\bar{z}^\gamma, \bar{u}^\gamma)$ and $(\bar{z}_h^\gamma, \bar{u}_h^\gamma)$ be solutions (controls and states) to \eqref{eq:regrpDir} and \eqref{eq:fullDisc} respectively.  There exists a positive constant $\widetilde{C}$ independent of $\gamma $ and $h$ such that 
\[
\max\left\{\|(\hat{\mu} + \gamma \bar{u}^\gamma)_+\|_{L^1(\Omega)}, \|(\hat{\mu} + \gamma \bar{u}_h^\gamma)_+\|_{L^1(\Omega)}\right\} \leq \widetilde{C}. 
\]
\end{lemma}
\begin{proof}
Let $\hat{z} \in Z_{ad}$ and $\hat{u} := \boldsymbol S\hat{z}$ satisfy the Slater condition such that
\[
\hat{u}(x)  \leq -\tau \qquad \forall x \in \overline{\Omega}, 
\]
where $\tau >0$ is a constant.  Using this we obtain 
\begin{alignat*}{4}
\int_\Omega (\hat{\mu} + \gamma \bar{u}^\gamma)_+ \tau \leq & - \int_\Omega (\hat{\mu} + \gamma \bar{u}^\gamma)_+ \hat{u} \\
= &-\mathcal E(\bar{\xi}^\gamma, \hat{u}) + (\bar{u}^\gamma - u_d, \hat{u})_{L^2(\Omega)} \tag{by \eqref{eq:regb}}\\
 = &-(\bar{\xi}^\gamma, \hat{z})_{L^2(\Omega)} +  (\bar{u}^\gamma - u_d, \hat{u})_{L^2(\Omega)} \tag{by \eqref{eq:rega}}\\
 \leq &-(\bar{\xi}^\gamma, \bar{z}^\gamma)_{L^2(\Omega)} - \alpha \|\bar{z}^\gamma\|_{L^2(\Omega)}^2 \\
 &+ \alpha( \bar{z}^\gamma, \hat{z})_{L^2(\Omega)} + (\bar{u}^\gamma - u_d, \hat{u})_{L^2(\Omega)} \tag{by \eqref{eq:regc}}\\
\leq  &-(\bar{\xi}^\gamma, \bar{z}^\gamma)_{L^2(\Omega)} + \alpha( \bar{z}^\gamma, \hat{z})_{L^2(\Omega)} + (\bar{u}^\gamma - u_d, \hat{u})_{L^2(\Omega)}\\
= &- \mathcal E(\bar{\xi}^\gamma, \bar{u}^\gamma) + \alpha( \bar{z}^\gamma, \hat{z})_{L^2(\Omega)} + (\bar{u}^\gamma - u_d, \hat{u})_{L^2(\Omega)} \tag{by \eqref{eq:rega}}\\
 = &- ((\hat{\mu} + \gamma \bar{u}^\gamma)_+, \bar{u}^\gamma)_{L^2(\Omega)} + (\bar{u}^\gamma - u_d, \bar{u}^\gamma)_{L^2(\Omega)} \\
& + \alpha( \bar{z}^\gamma, \hat{z})_{L^2(\Omega)} + (\bar{u}^\gamma - u_d, \hat{u})_{L^2(\Omega)} \tag{by \eqref{eq:regb}}.
\end{alignat*}
Recalling the definition of $\Omega_\gamma^+(\hat{\mu})$ from \eqref{eq:4f}, we have 
\begin{alignat*}{4}
\int_\Omega (\hat{\mu} + \gamma \bar{u}^\gamma)_+ \tau \leq & - \int_{\Omega_\gamma^+(\hat{\mu})} (\hat{\mu} + \gamma \bar{u}^\gamma) \bar{u}^\gamma + \alpha(\bar{z}^\gamma,  \hat{z})_{L^2(\Omega)}\\
& + (\bar{u}^\gamma - u_d, \hat{u} + \bar{u}^\gamma)_{L^2(\Omega)}\\
=&-\int_{\Omega_\gamma^+(\hat{\mu})} \hat{\mu}\bar{u}^\gamma - \gamma \|\bar{u}^\gamma\|_{L^2(\Omega_\gamma^+(\hat{\mu}))}^2 + \alpha(\bar{z}^\gamma,  \hat{z})_{L^2(\Omega)}  \\
& + (\bar{u}^\gamma - u_d, \hat{u} + \bar{u}^\gamma)_{L^2(\Omega)}\\
\leq & -\int_{\Omega_\gamma^+(\hat{\mu})} \hat{\mu}\bar{u}^\gamma + \alpha(\bar{z}^\gamma,  \hat{z})_{L^2(\Omega)}  + (\bar{u}^\gamma - u_d, \hat{u} + \bar{u}^\gamma)_{L^2(\Omega)}.
\end{alignat*}
The $L^1$ bound for the left-hand-side follows from the fact that we can bound all of the terms on the right-hand-side of the above inequality independent of $\gamma$ using \eqref{regsolbound}. 

To obtain the $L^1$ bound on $(\hat{\mu} + \gamma \bar{u}_h^\gamma)_+$, we take $\hat{z}$ and $\hat{u}$ as before, again use $\Pi_h$ as the $L^2$-orthogonal projection onto $Z_{ad,h}$, and let $\hat{u}_h$ be the solution of 
\[
\mathcal E(\hat{u}_h, v_h) = (\Pi_h \hat{z}, v_h)_{L^2(\Omega)} \qquad \qquad \forall v_h \in U_h, 
\]
so that 
\[
\mathcal E(\hat{u} - \hat{u}_h, v_h) = (\hat{z} - \Pi_h \hat{z}, v_h)_{L^2(\Omega)} \qquad \qquad \forall v_h \in U_h. 
\]
From \cite[Proposition 3.2]{HAntil_DVerma_MWarma_2019b}, we now have the bound
\[
\|\hat{u} - \hat{u}_h\|_{L^\infty(\Omega)} \leq C \|\hat{z} - \Pi_h \hat{z}\|_{L^p(\Omega)} \to 0 \qquad \text{as } h \to 0,
\]
and hence $\hat{u}_h$ converges to $\hat{u}$ uniformly.  Note that the convergence of $\|\hat{z} - \Pi_h \hat{z}\|_{L^p(\Omega)}$ follows from the density of piecewise constant functions in $L^p(\Omega)$, (every $f$ in $L^p(\Omega)$ can be written as the limit of a sequence of simple functions).  Therefore we have that $\hat{u}_h \leq  -  \hat{\tau}$ with $\hat{\tau} \leq \tau$ and we can use the same argument above using the discrete optimal solution $(\bar{z}_h^\gamma, \bar{u}_h^\gamma, \bar{\xi}_h^\gamma)$ and $\Pi_h\hat{z}$ instead of $\hat{z}$ to obtain the desired bound. 
\end{proof}  

We now present a convergence estimate for which the discrete control converges to the continuous control where all the terms (including constants) in the estimate remain bounded as $\gamma \to \infty$.
\begin{theorem}\label{thm:errBound2}
Assume that $\Omega$ has a $C^\infty$ boundary, $Z = L^2(\Omega)$ and therefore 
	$s> N/4$.  
	Let $\bar{z}^\ga$ and $\bar{z}_h^\ga$ denote the solution of \eqref{eq:regrpDir} and \eqref{eq:fullDisc} respectively with corresponding states $\bar{u}^\ga = \boldsymbol{S}\bar{z}^\gamma$ and $\bar{u}_h^\ga = \boldsymbol{S}_h \bar{z}^\gamma$.   Then the discretization error can be bounded as 
\begin{alignat*}{4}
\|\bar{z} - \bar{z}_h^\gamma\|_{L^2(\Omega)} \leq C&\bigg(\| (\bar{u}^\gamma)_+ - I_h(\bar{u}^\gamma)_+ \|_{C_0(\Om)} + \gamma^{-1/2} h^{-N/2}\\
& + \|R-R_h\|_{L^1 \to L^2}   + h^\beta\|\bar{z}^\gamma\|_{L^2(\Omega)}\\
& + \|(\hat{\mu} + \gamma \boldsymbol S \bar{z}^\gamma)_+ - (\hat{\mu} + \gamma \boldsymbol S_h \bar{z}^\gamma)_+ \|_{L^1(\Omega)}\\
& +  \|\bar{z}^\gamma - \Pi_h \bar{z}^\gamma\|_{L^2(\Omega)} +  \|\bar{z}^\gamma - \Pi_h \bar{z}^\gamma\|_{L^2(\Omega)}^{1/2}\bigg), 
\end{alignat*}
where the positive constant $C$ is independent of $\gamma$ and $h$, and $\beta = \min\{2s, 1-\varepsilon\}$. 
\end{theorem}
\begin{proof}
This proof involves many of the same steps as the proof of Theorem \ref{thm:errBound1}.  As such, we will use the same notation and only highlight the differences needed for this result.  Recall the terms in the decomposition \eqref{eq:decomp1}.  We bound  $(III)$ exactly as before.  For $(I)$ recall  
\[
	(I) \leq C \|(R - R_h) f(\boldsymbol S \bar{z}^\gamma)\| \|\bar{z}_h^\gamma - \bar{z}^\gamma\|. 
\]
We will now consider the operators $R$ and $R_h$ acting on data in $L^1(\Omega)$, and note that the operators $R_h$ are uniformly bounded as operators from $L^1(\Omega)$ to $L^2(\Omega)$ (see \cite[Theorem 3.6]{HAntil_DVerma_MWarma_2019b}). 
Furthermore, from \cite[Proposition 3.8]{BoBoNoOtSa2018}, we have the estimate 
\begin{equation} \label{eq:hRate}
\|(R - R_h) v \|_{L^2(\Omega)} \leq C h^\beta \|v\|_{L^2(\Omega)} \qquad \forall v \in L^2(\Omega),
\end{equation}
where $\beta = \min\{2s, 1-\varepsilon\}$ so that $\|R-R_h\|_{L^2 \to L^2} \to 0$ as $h \to 0$.  Using the density of $L^2(\Omega)$ in $L^1(\Omega)$ and the fact that $\|f(\boldsymbol S\bar{z}^\gamma)\|_{L^1(\Omega)}$ is uniformly bounded (by \eqref{eq:rel} and Lemma \ref{lem:L1bound}), we have that $\|R-R_h\|_{L^1 \to L^2} \to 0$ as $h\to 0$, as a consequence of the Banach-Steinhaus Theorem (see \cite[Corollary 9.2]{Limaye1996}).  Therefore we have 
\[
\|(R - R_h) f(\boldsymbol S \bar{z}^\gamma)\|  \leq C\|R-R_h\|_{L^1 \to L^2} \|f(S\bar{z}^\gamma)\|_{L^1(\Omega)} \to 0,
\]
as $h\rightarrow 0$ and by appealing to Lemma \ref{lem:L1bound}, we have that this convergence is independent of $\gamma$.  Similarly in $(II)$, we can bound
\begin{alignat*}{4}
\|R_h(f(\bar{u}^\gamma) - f(\hat{u}_h^\gamma))\| \leq \hspace{-1pt}  C &\|(\boldsymbol{S} - \boldsymbol{S}_h) \bar{z}^\gamma \hspace{-1pt} + (\hat{\mu} + \gamma \boldsymbol S \bar{z}^\gamma)_+ \hspace{-2pt} - (\hat{\mu} + \gamma \boldsymbol S_h \bar{z}^\gamma)_+ \hspace{-1pt} \|_{L^1(\Omega)} \\
\leq C&\big(\|(\boldsymbol{S} - \boldsymbol{S}_h) \bar{z}^\gamma\|_{L^2(\Omega)}\\
& + \|(\hat{\mu} + \gamma \boldsymbol S \bar{z}^\gamma)_+ - (\hat{\mu} + \gamma \boldsymbol S_h \bar{z}^\gamma)_+ \|_{L^1(\Omega)} \big)\\
\leq  C &\big( h^\beta \|\bar{z}^\gamma\|_{L^2(\Omega)}\\
& + \|(\hat{\mu} + \gamma \boldsymbol S \bar{z}^\gamma)_+ - (\hat{\mu} + \gamma \boldsymbol S_h \bar{z}^\gamma)_+ \|_{L^1(\Omega)}\big),
\end{alignat*}
and the second term converges independent of $\gamma$ since $(\hat{\mu}+ \gamma \boldsymbol S_h\bar{z}^\gamma)_+ \to (\hat{\mu} +  \gamma \boldsymbol S \bar{z}^\gamma)_+$ pointwise almost everywhere in $\Omega$ and $(\hat{\mu} +  \gamma \boldsymbol S \bar{z}^\gamma)_+$ is bounded independent of $\gamma$ in $L^1(\Omega)$ by Lemma \ref{lem:L1bound}.  Now recall the decomposition in \eqref{eq:decomp2}.  We bound $(i)$ and $(ii)$ similarly to $(I)$ and $(II)$.  As with $(I)$ and $(II)$ above, we bound $(iii)$ as
\begin{alignat*}{4}
(iii) & = (R_h f(\hat{u}_h^\gamma) - R_h f (\bar{u}_h^\gamma),  \bar{z}^\gamma - \Pi_h \bar{z}^\gamma)\\
&  \leq C \|R_h(f(\hat{u}_h^\gamma) - f (\bar{u}_h^\gamma))\| \|\bar{z}^\gamma - \Pi_h \bar{z}^\gamma\|\\
&\leq C \|\boldsymbol{S}_h( \bar{z}^\gamma  - \bar{z}_h^\gamma) + (\hat{\mu} + \gamma \boldsymbol S_h \bar{z}^\gamma)_+ - (\hat{\mu} + \gamma \boldsymbol S_h \bar{z}_h^\gamma)_+\|_{L^1(\Omega)} \|\bar{z}^\gamma - \Pi_h \bar{z}^\gamma\|. 
\end{alignat*}
We now appeal to Lemma \ref{lem:L1bound} once again and the boundedness of $\boldsymbol S_h$ to obtain 
\begin{alignat*}{4}
(iii) \leq  C\Big( &\hat{C}\|\bar{z}^\gamma  -  \bar{z}_h^\gamma\|_{L^2(\Omega)} + \|(\hat{\mu} + \gamma \boldsymbol S_h \bar{z}^\gamma)_+\|_{L^1(\Omega)}\\
&+ \|(\hat{\mu} + \gamma \boldsymbol S_h \bar{z}^\gamma)_+\|_{L^1(\Omega)}\Big) \|\bar{z}^\gamma - \Pi_h \bar{z}^\gamma\|\\
\leq C \Big( &\hat{C}\|\bar{z}^\gamma  -  \bar{z}_h^\gamma\|_{L^2(\Omega)} + \widetilde{C}\Big) \|\bar{z}^\gamma - \Pi_h \bar{z}^\gamma\|.
\end{alignat*}
For $(iv)$ we obtain the simple bound 
\[
(iv)  \leq \|\Pi_h \bar{\xi}^\gamma  - \bar{\xi}^\gamma\| \|\bar{z}^\gamma - \Pi_h \bar{z}^\gamma\|.
\]
Using \eqref{regsolbound} and Lemma \ref{lem:L1bound} we can further bound
\begin{alignat*}{3}
\|\Pi_h \bar{\xi}^\gamma  - \bar{\xi}^\gamma\| \leq &C \|\bar{\xi}^\gamma\| \quad \Big(= \|Rf(\boldsymbol S\bar{z}^\gamma)\| \Big)\\
\leq &C \|\boldsymbol{S} \bar{z}^\gamma -  u_d + (\hat{\mu} + \gamma \boldsymbol S \bar{z}^\gamma)_+\|_{L^1(\Omega)}\\
\leq & C(C_{\bar{z}} + \widetilde{C}),
\end{alignat*}
and this is independent of $\gamma$.  Combining the preceding bounds, using Young's inequality as in the proof of Theorem \ref{thm:errBound1}, and making use of \eqref{eq:hRate} we can replace \eqref{eq:errBound} with  
\begin{alignat*}{5}
\|\bar{z}^\gamma - \bar{z}_h^\gamma\| \leq  C \Bigg(&\|R-R_h\|_{L^1 \to L^2}\|f(\boldsymbol S(\bar{z}^\gamma)\|_{L^1(\Omega)} + h^\beta\|\bar{z}^\gamma\|_{L^2(\Omega)}\\
& + \|(\hat{\mu} + \gamma \boldsymbol S \bar{z}^\gamma)_+ - (\hat{\mu} + \gamma \boldsymbol S_h \bar{z}^\gamma)_+ \|_{L^1(\Omega)}\\
& +  \|\bar{z}^\gamma - \Pi_h \bar{z}^\gamma\|_{L^2(\Omega)} +  \|\bar{z}^\gamma - \Pi_h \bar{z}^\gamma\|_{L^2(\Omega)}^{1/2}\Bigg), 
\end{alignat*}
from which the result follows.
\end{proof}


\section{Numerical Results}
\label{s:numerics}

In this section we present some numerical results to reinforce the efficacy of our approach.  
We set $\Omega \subset \mathbb R^2$ to be a disk of radius $1/2$ centered at the origin.  
We truncate $\mathbb R^2\backslash \Omega$ and bound the exterior of $\Omega$ with a circle of radius $3/2$ centered at the origin.  We use standard $\mathcal P_1$ Lagrangian finite elements for our spatial discretization of the state and adjoint and a sequence of 8 increasingly finer meshes.  We will refer to the meshes by number with mesh 1 being the coarsest and mesh 8 being the finest.  For details on the discretization of the fractional Laplacian operator see \cite{acosta2017short}.  To minimize the objective function, we use the BFGS method. 

In all of the experiments we define our desired state to be 
\[
u_d (x,y) = \frac{2^{-2s}}{\Gamma(1+s)^2}(1/4-(x^2 + y^2)_+)^s.
\]
For $s = 0.2$, the desired state is shown in the top left panel in Figure \ref{fig:3}.  We take $u_b = 0.1$ as our state constraint.  In the left panel in Figure \ref{fig:2} we show the convergence of $\|(u-u_b)_+\|_{L^2(\Omega)}$ for $s = 0.4$ on the finest mesh as $\gamma$ increases.  We note that since we are taking our desired state to be non-negative we get full $\mathcal O(\gamma^{-1})$ convergence as stated in Theorem \ref{ydorder}.  We also show the maximum error in  $u - u_b$ with respect to $\gamma$ on mesh 5 in the right panel of Figure \ref{fig:2}.

\begin{figure}[ht]
\centering
\includegraphics[width=0.45\textwidth]{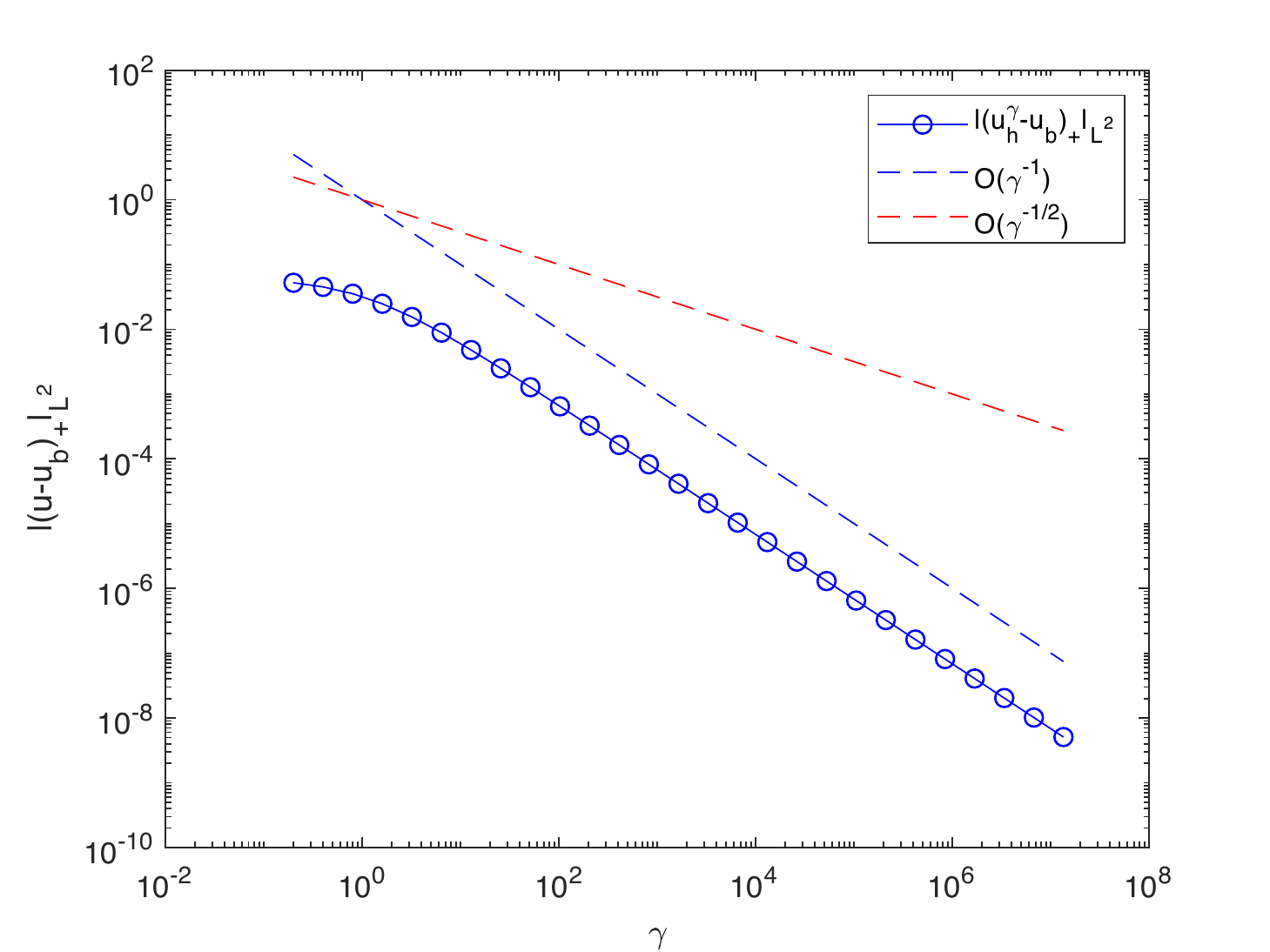} \quad
\includegraphics[width=0.45\textwidth]{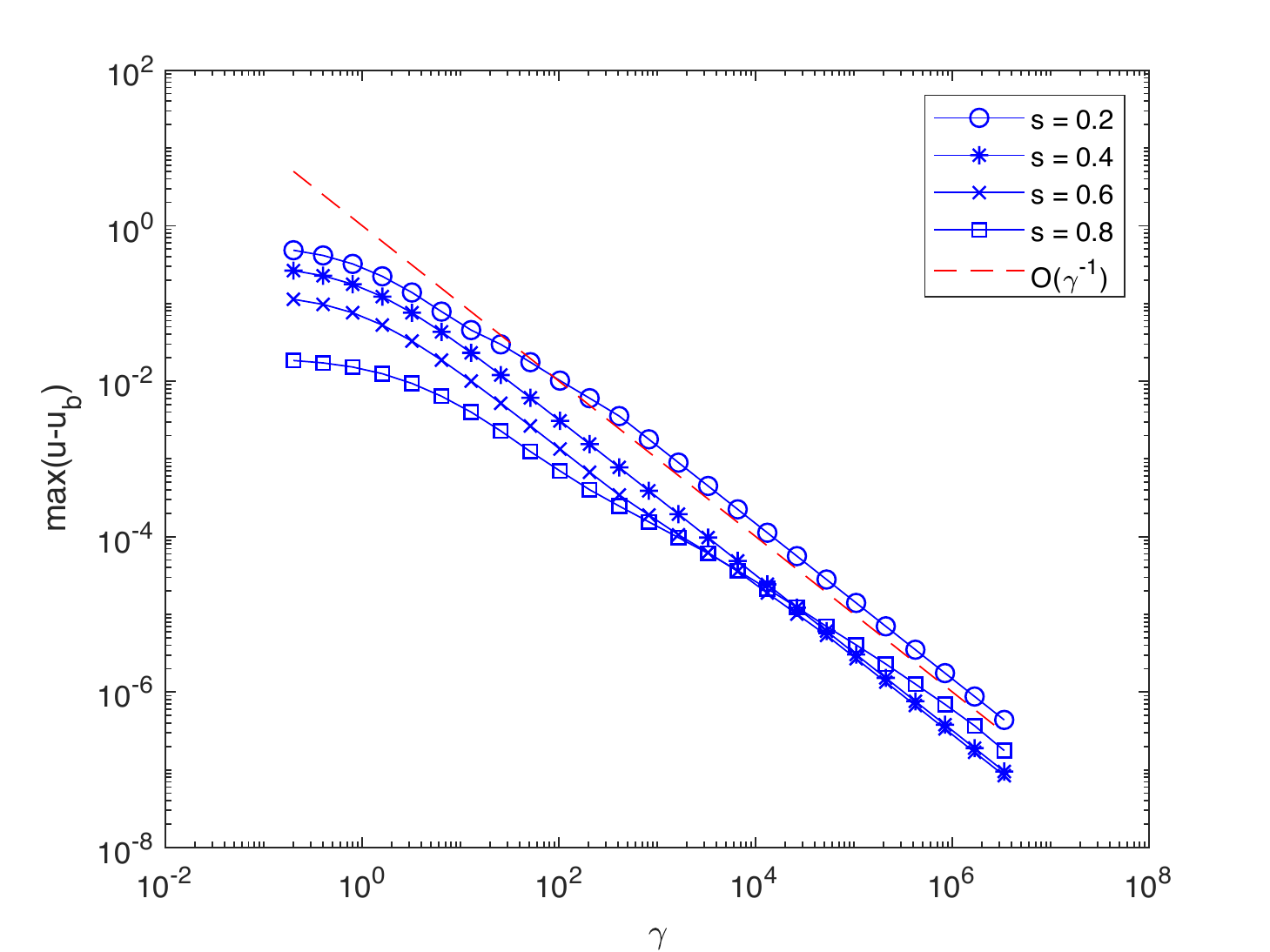}
\caption{Convergence of $\|(u-u_b)_+\|_{L^2(\Omega)}$ (left) and the violation of the state constraints in the max-norm ($\max(u - u_b)$) for given fractional exponent $s$ (right) as $\gamma$ increases.}\label{fig:2}
\end{figure}

In Figure \ref{fig:2.5} we show some convergence results for $s = 0.2$ for the control and the state as $\gamma$ increases.  Since we do not have an exact solution for the control problem and the associated optimal state, obtaining error estimates for the optimal control and state is problematic.  To obtain the plots in Figure \ref{fig:2.5} we solve the optimal control problem on the finest mesh (mesh 8), then project this solution to the coarse meshes.  We then take the $L^2$ error between solutions on the four coarsest meshes and the projected solution.

\begin{figure}[ht]
\centering
\includegraphics[width =\textwidth]{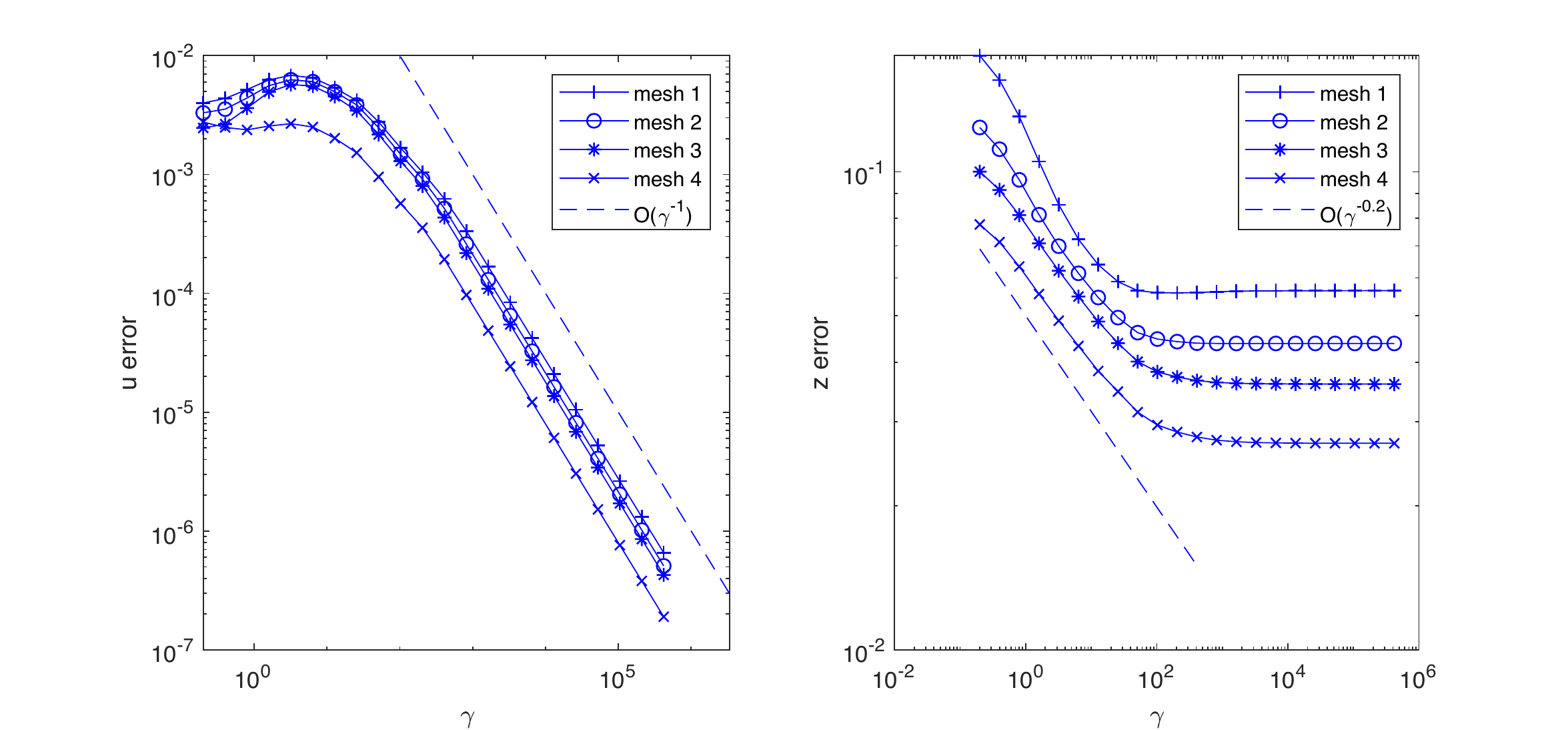}
\caption{
The error $\|u - u^\gamma_h\|_{L^2(\Omega)}$ (left) and $\|z - z^\gamma_h\|_{L^2(\Omega)}$ (right)
with respect to $\gamma$ for various values of $h$ (different meshes). We recall that $(u,z)$ are computed by 
solving the optimal control problem on mesh 8.
} \label{fig:2.5}
\end{figure}

In Figure \ref{fig:3} we show the optimal state, control, and Lagrange multiplier for $s = 0.2$ and $\gamma = 419430.4$.  We note that the control is a piecewise constant on the mesh.  The optimal state in Figure \ref{fig:3} appears to be cleanly cut off at $u_b = 0.1$, complying with the state constraints and resulting in a cylindrical profile. Moreover, notice that the Lagrange multiplier corresponding to the inequality constraints is a measure (bottom right panel) as expected.  

\begin{figure}[ht]
\centering
\includegraphics[width=0.4 \textwidth]{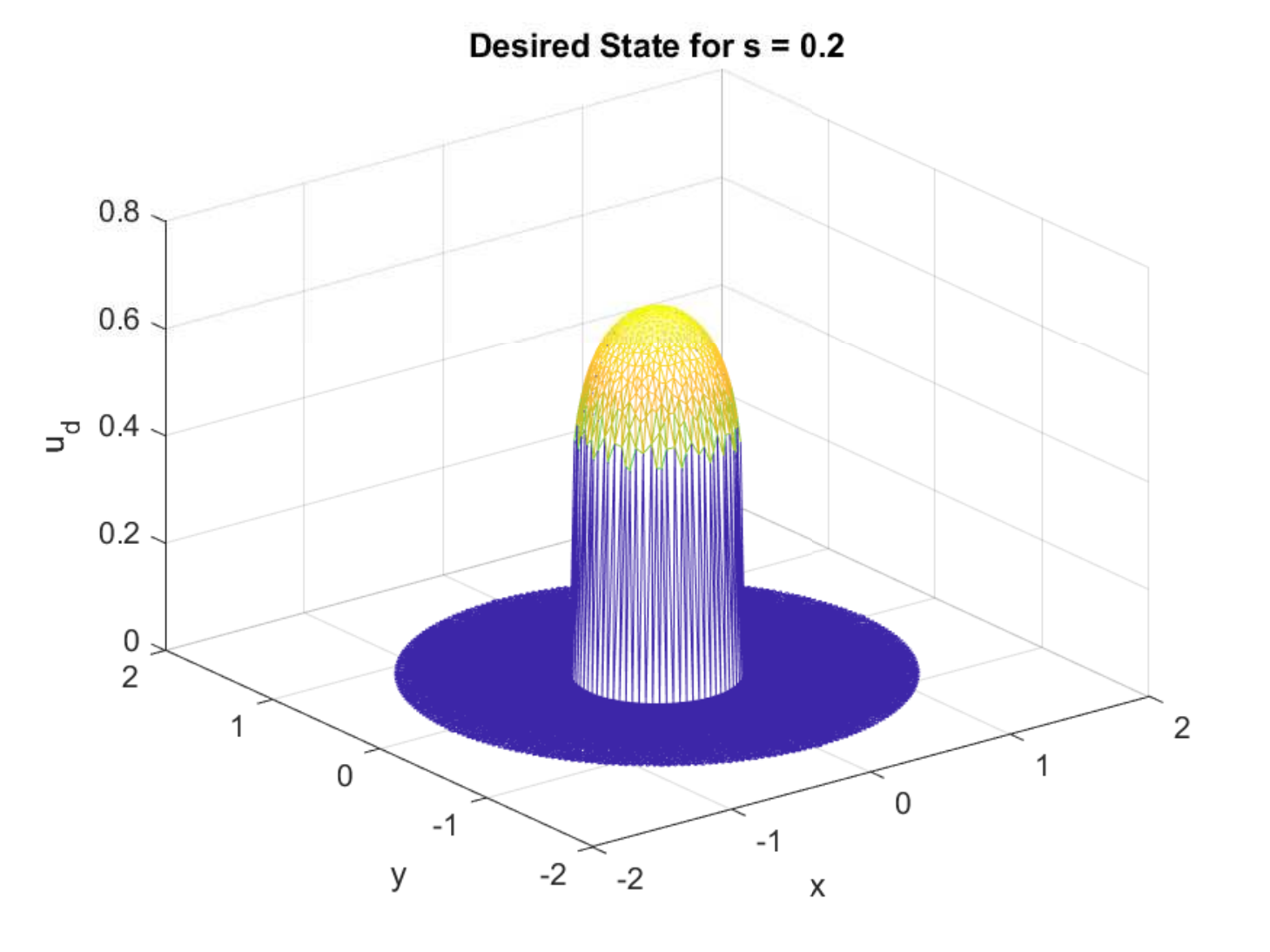} \qquad 
\includegraphics[width=0.4\textwidth]{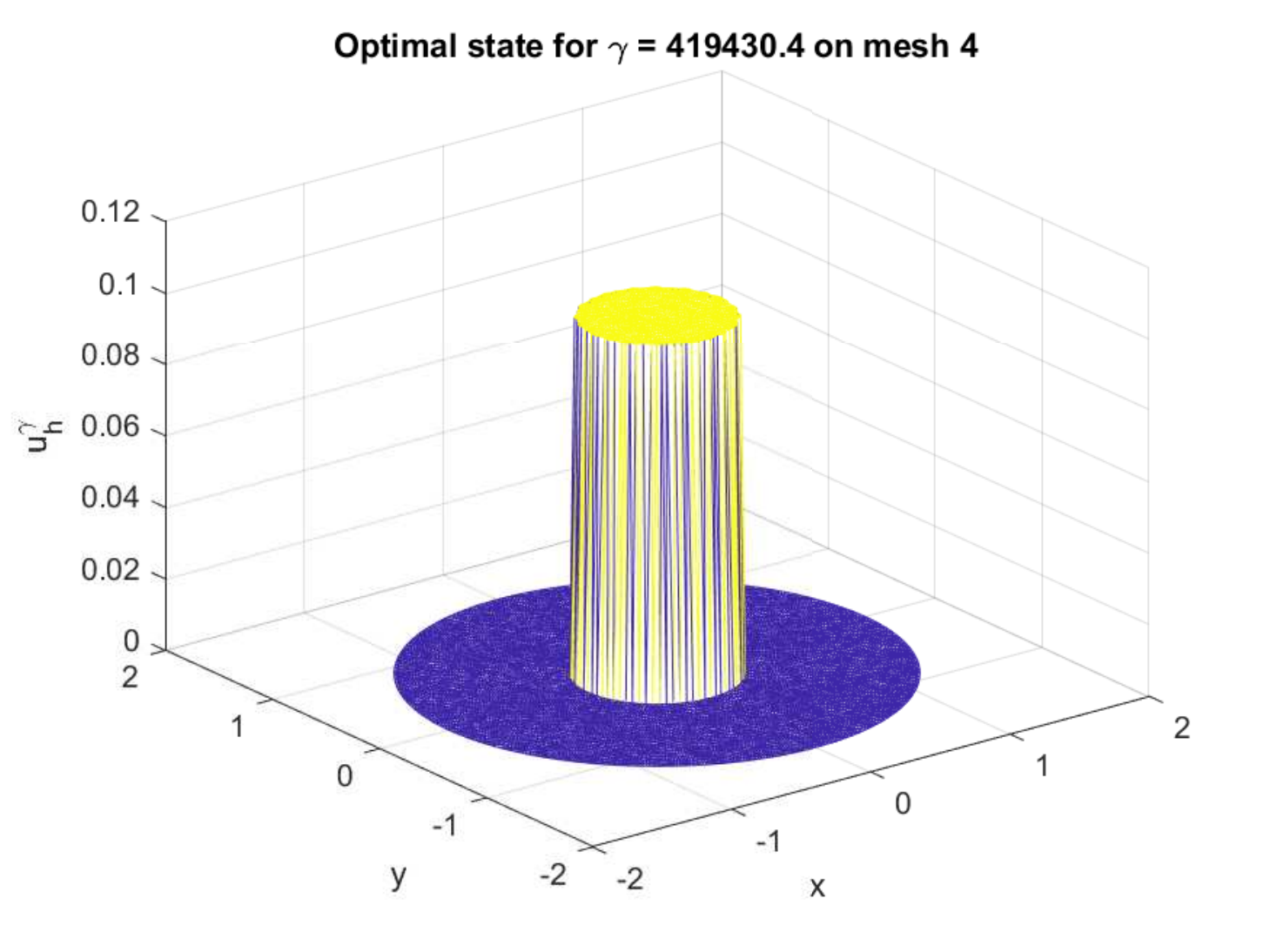}\\
\includegraphics[width=0.4\textwidth]{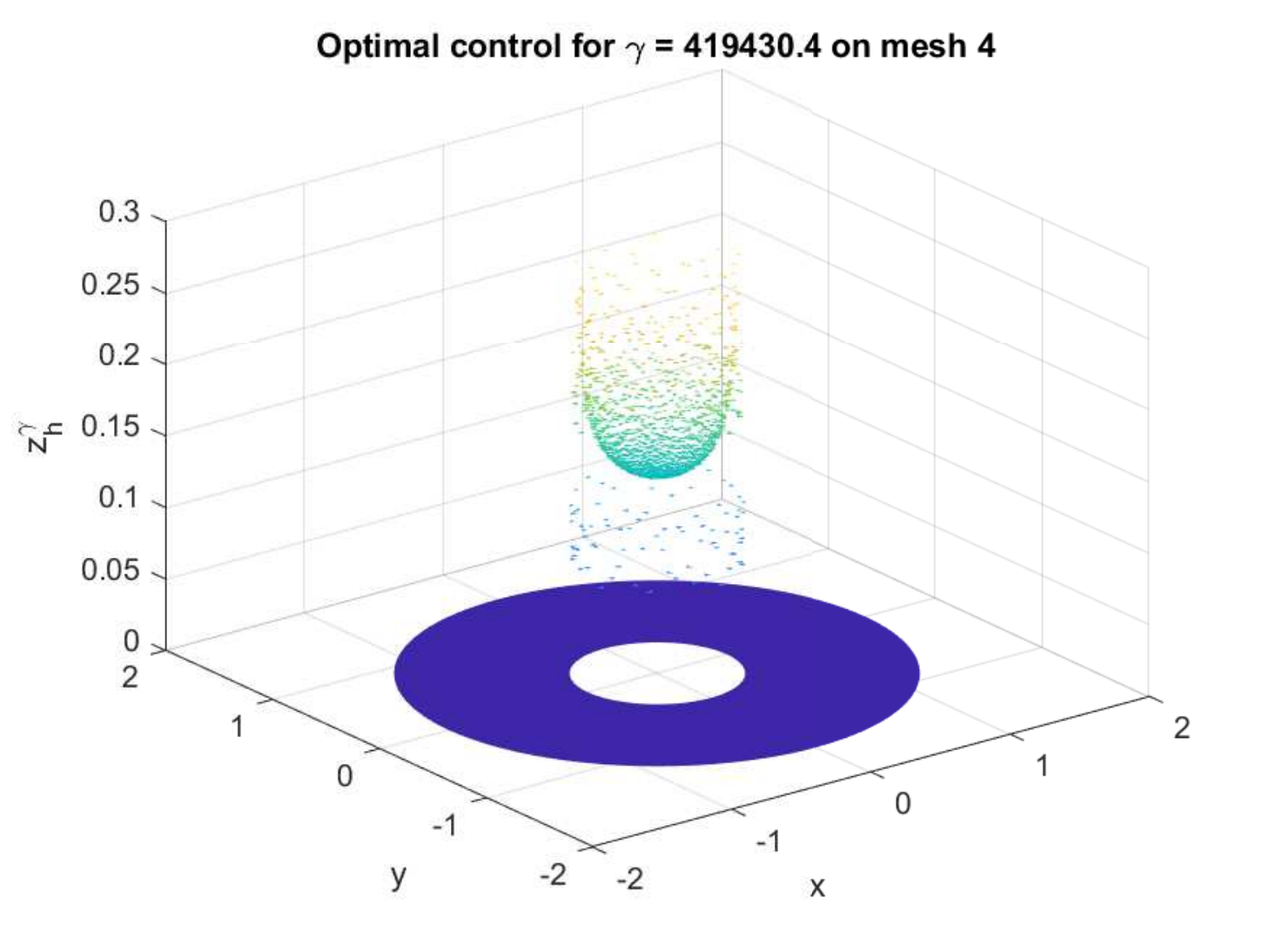} \qquad
\includegraphics[width=0.4\textwidth]{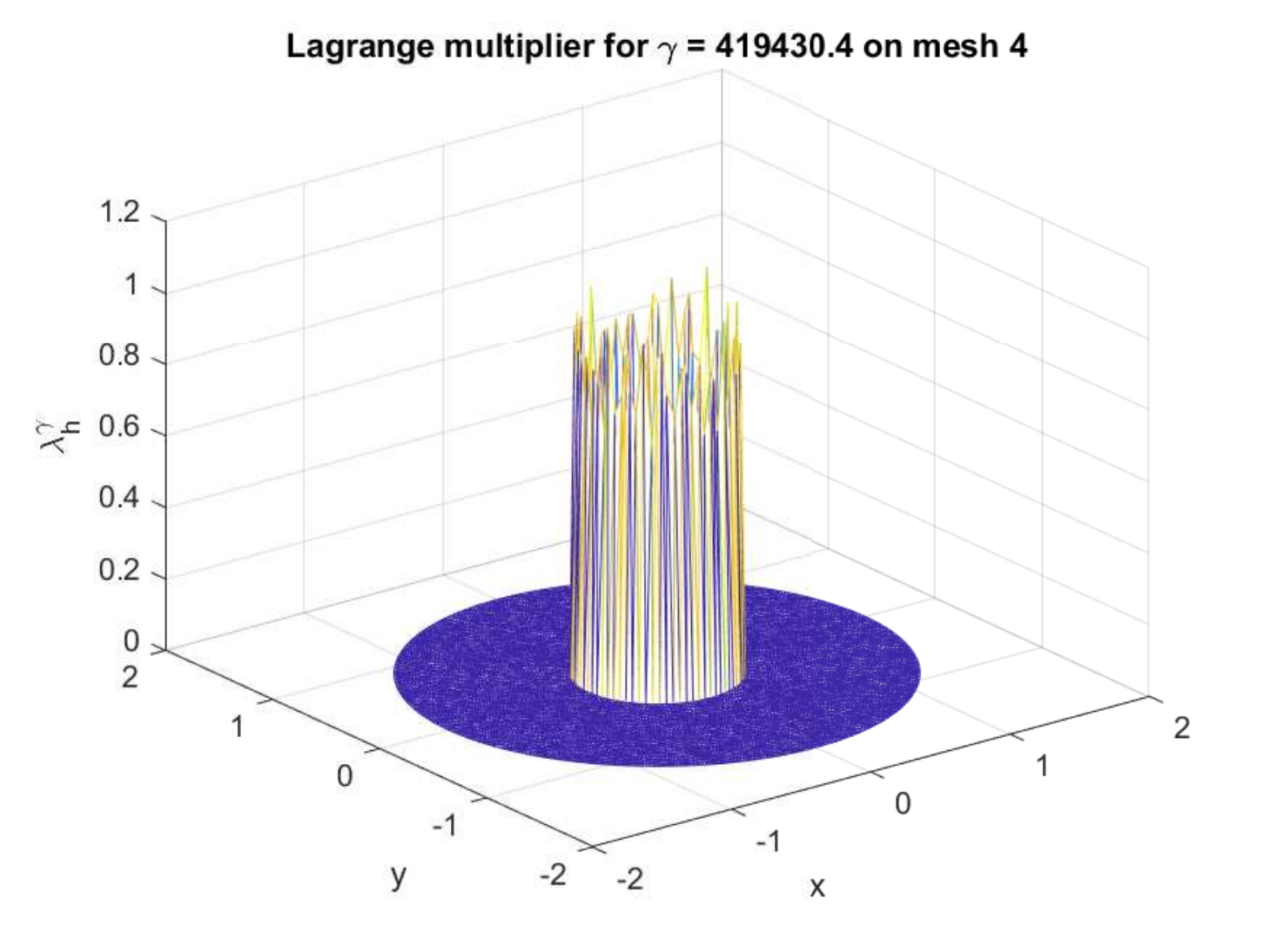}
\caption{The desired state (top left) optimal state (top right),  control (bottom left) and Lagrange multiplier (bottom right) for $s = 0.2$ and $\gamma = 419430.4$.}\label{fig:3}
\end{figure}

\section{Conclusions}

We have presented a Moreau-Yosida regularized version of the optimal control problem considered in \cite{HAntil_DVerma_MWarma_2019b}. While this regularization has been used and studied in other contexts, this is the first work to do so for optimal control problems constrained by fractional PDE with state constraints.  We have also provided a convergence analysis for the solution of the regularized problem to the solution of the original problem as well as analysis of the solution of a discretized version of the problem to the original solution.  The discretized version of the problem was implemented using a finite element method.  Our numerical experiments validate the theoretical findings and show that the Moreau-Yosida regularization is an effective method for solving problems involving the optimal control of fractional elliptic PDE with state constraints.  

We have established the convergence of the discrete approximation of the optimal control to original control.  Furthermore, for a fixed $\gamma$ we have established a convergence rate in $h$. In addition, for a range of $s$ values, we have shown the convergence of our error bound independent of $\gamma$.


\bibliographystyle{plain}
\bibliography{refs}

\def\ocirc#1{\ifmmode\setbox0=\hbox{$#1$}\dimen0=\ht0 \advance\dimen0
  by1pt\rlap{\hbox to\wd0{\hss\raise\dimen0
  \hbox{\hskip.2em$\scriptscriptstyle\circ$}\hss}}#1\else {\accent"17 #1}\fi}
  \def\cprime{$'$} \def\ocirc#1{\ifmmode\setbox0=\hbox{$#1$}\dimen0=\ht0
  \advance\dimen0 by1pt\rlap{\hbox to\wd0{\hss\raise\dimen0
  \hbox{\hskip.2em$\scriptscriptstyle\circ$}\hss}}#1\else {\accent"17 #1}\fi}
\begin{thebibliography}{10}

\bibitem{acosta2017short}
G.~Acosta, F.M. Bersetche, and J.P. Borthagaray.
\newblock A short fe implementation for a 2d homogeneous dirichlet problem of a
  fractional laplacian.
\newblock {\em Computers \& Mathematics with Applications}, 74(4):784--816,
  2017.

\bibitem{HAntil_CNRautenberg_2019b}
H.~Antil and C.N. Rautenberg.
\newblock Sobolev spaces with non-{M}uckenhoupt weights, fractional elliptic
  operators, and applications.
\newblock {\em SIAM J. Math. Anal.}, 51(3):2479--2503, 2019.

\bibitem{HAntil_DVerma_MWarma_2019b}
H.~Antil, D.~Verma, and M.~Warma.
\newblock Optimal control of fractional elliptic {PDE}s with state constraints
  and characterization of the dual of fractional order sobolev spaces.
\newblock {\em Submitted}, 2019.

\bibitem{antil2018b}
H.~Antil and M.~Warma.
\newblock Optimal control of the coefficient for fractional $\{$$ p
  $$\}$-$\{$L$\}$ aplace equation: Approximation and convergence.
\newblock {\em RIMS K\^{o}ky\^{u}roku}, 2090:102--116, 2018.

\bibitem{antil2017optimal}
H.~Antil and M.~Warma.
\newblock Optimal control of fractional semilinear {PDE}s.
\newblock {\em To appear: Control, Optimisation and Calculus of Variations
  (ESAIM: COCV)}, 2019.

\bibitem{BoBoNoOtSa2018}
A~Bonito, J.~P. Borthagaray, R.~H. Nochetto, E.~Ot\'{a}rola, and A.~J. Salgado.
\newblock Numerical methods for fractional diffusion.
\newblock {\em Comput. Vis. Sci.}, 19(5-6):19--46, 2018.

\bibitem{Caf3}
L.~Caffarelli and L.~Silvestre.
\newblock An extension problem related to the fractional {L}aplacian.
\newblock {\em Comm. Partial Differential Equations}, 32(7-9):1245--1260, 2007.

\bibitem{MDElia_MGunzburger_2014a}
M.~D'Elia and M.~Gunzburger.
\newblock Optimal distributed control of nonlocal steady diffusion problems.
\newblock {\em SIAM J. Control Optim.}, 52(1):243--273, 2014.

\bibitem{NPV}
E.~Di~Nezza, G.~Palatucci, and E.~Valdinoci.
\newblock Hitchhiker's guide to the fractional {S}obolev spaces.
\newblock {\em Bull. Sci. Math.}, 136(5):521--573, 2012.

\bibitem{SDipierro_XRosOton_EValdinoci_2017a}
S.~Dipierro, X.~Ros-Oton, and E.~Valdinoci.
\newblock Nonlocal problems with {N}eumann boundary conditions.
\newblock {\em Rev. Mat. Iberoam.}, 33(2):377--416, 2017.

\bibitem{ErGu2004}
A~Ern and J.-L. Guermond.
\newblock {\em Theory and practice of finite elements}, volume 159 of {\em
  Applied Mathematical Sciences}.
\newblock Springer-Verlag, New York, 2004.

\bibitem{GGrubb_2015a}
G.~Grubb.
\newblock Fractional {L}aplacians on domains, a development of {H}\"ormander's
  theory of {$\mu$}-transmission pseudodifferential operators.
\newblock {\em Adv. Math.}, 268:478--528, 2015.

\bibitem{HiHi2009}
M.~Hinterm\"{u}ller and M.~Hinze.
\newblock Moreau-{Y}osida regularization in state constrained elliptic control
  problems: error estimates and parameter adjustment.
\newblock {\em SIAM J. Numer. Anal.}, 47(3):1666--1683, 2009.

\bibitem{hintermuller_kunisch}
M.~Hinterm\"{u}ller and K.~Kunisch.
\newblock Feasible and noninterior path-following in constrained minimization
  with low multiplier regularity.
\newblock {\em SIAM J. Control Optim.}, 45(4):1198--1221, 2006.

\bibitem{HiScWo2014}
M.~Hinterm\"{u}ller, A.~Schiela, and W.~Wollner.
\newblock The length of the primal-dual path in {M}oreau-{Y}osida-based
  path-following methods for state constrained optimal control.
\newblock {\em SIAM J. Optim.}, 24(1):108--126, 2014.

\bibitem{KIto_KKunisch_2008a}
K.~Ito and K.~Kunisch.
\newblock {\em Lagrange multiplier approach to variational problems and
  applications}, volume~15 of {\em Advances in Design and Control}.
\newblock Society for Industrial and Applied Mathematics (SIAM), Philadelphia,
  PA, 2008.

\bibitem{KaStWa2018}
C.~Kanzow, D.~Steck, and D.~Wachsmuth.
\newblock An augmented {L}agrangian method for optimization problems in
  {B}anach spaces.
\newblock {\em SIAM J. Control Optim.}, 56(1):272--291, 2018.

\bibitem{KeUl2015}
M.~Keuthen and M.~Ulbrich.
\newblock Moreau-{Y}osida regularization in shape optimization with geometric
  constraints.
\newblock {\em Comput. Optim. Appl.}, 62(1):181--216, 2015.

\bibitem{KuWa2012}
K.~Kunisch and D.~Wachsmuth.
\newblock Path-following for optimal control of stationary variational
  inequalities.
\newblock {\em Comput. Optim. Appl.}, 51(3):1345--1373, 2012.

\bibitem{Limaye1996}
B.~V. Limaye.
\newblock {\em Functional analysis}.
\newblock New Age International Publishers Limited, New Delhi, second edition,
  1996.

\bibitem{MeYo2009}
C.~Meyer and I.~Yousept.
\newblock Regularization of state-constrained elliptic optimal control problems
  with nonlocal radiation interface conditions.
\newblock {\em Comput. Optim. Appl.}, 44(2):183--212, 2009.

\bibitem{TrVo2001}
F.~Tr\"{o}ltzsch and S.~Volkwein.
\newblock The {SQP} method for control constrained optimal control of the
  {B}urgers equation.
\newblock {\em ESAIM Control Optim. Calc. Var.}, 6:649--674, 2001.

\bibitem{UlUlBr2017}
M.~Ulbrich, S.~Ulbrich, and D.~Bratzke.
\newblock A multigrid semismooth {N}ewton method for semilinear contact
  problems.
\newblock {\em J. Comput. Math.}, 35(4):486--528, 2017.

\bibitem{War}
M.~Warma.
\newblock The fractional relative capacity and the fractional {L}aplacian with
  {N}eumann and {R}obin boundary conditions on open sets.
\newblock {\em Potential Anal.}, 42(2):499--547, 2015.

\bibitem{CWeiss_BvBWaanders_HAntil_2018a}
C.~Weiss, B.~van~Bloemen Waanders, and H.~Antil.
\newblock Fractional operators applied to geophysical electromagnetics.
\newblock {\em To appear: Geophysical Journal International}, 2019.

\end{thebibliography}

\end{document}